\documentclass[reqno,11pt]{amsart}
\usepackage[utf8]{inputenc}
\usepackage[T1]{fontenc}
\usepackage{geometry}
\usepackage{graphicx} 
\usepackage{amsmath}
\usepackage{mathtools}
\usepackage{amsthm}
\usepackage{amssymb}
\usepackage{hyperref}

\newcommand{\n}{\nabla}
\newcommand{\fk}{\frac{\mu_{t}^{k,N}}{\bar\mu_{t}^{\otimes k}}}
\newcommand{\norm}[1]{\left\lVert#1\right\rVert}
\newcommand{\mk}{\mu_{t}^{k,N}}
\newcommand{\mtk}{\bar\mu_{t}^{\otimes k}}
\usepackage{amsfonts}
\usepackage{xcolor}

\keywords{Fisher information, propagation of chaos, sharp rates}

\subjclass[2020]{65C35; 35K55; 65C05; 82C22; 26D10; 60E15}

\author{Jules~Grass}
\address{Université Claude Bernard Lyon 1, CNRS UMR 5208, Institut Camille Jordan, 69622 Villeurbanne, France. \url{grass@math.univ-lyon1.fr}}
\author{Arnaud~Guillin}
\address{Laboratoire de Mathématique Blaise Pascal, CNRS UMR 6620, Université Clermont-Auvergne, avenue des Landais, F-63177 Aubière. \url{arnaud.guillin@uca.fr}}
\author{Christophe~Poquet}
\address{Université Claude Bernard Lyon 1, CNRS UMR 5208, Institut Camille Jordan, 69622 Villeurbanne, France. \url{poquet@math.univ-lyon1.fr}}

\title{Propagation of chaos in Fisher information}

\newtheorem{theorem}{Theorem}[section]
\newtheorem{lemma}[theorem]{Lemma}
\newtheorem{proposition}[theorem]{Proposition}
\newtheorem{hyp}[theorem]{Hypothesis}

\newtheorem{rem}[theorem]{Remark}

\newcommand{\cp}[1]{{\color{blue}#1}}

\begin{document}

\maketitle

\begin{abstract}
We present a new method for proving sharp local propagation of chaos in Fisher Information for particles in $\mathbb{R}^{d}$ with smooth interaction and drift. We rely on a new Lemma computing the Fisher Information of two diffusion processes with smooth drifts and fine estimates on the hessian of the law of the solution of the McKean-Vlasov equation. It allows us to obtain a new propagation of chaos in Fisher information, generalizing Lacker's seminal work \cite{lacker}, by using the BBGKY hierarchy to obtain a system of differential inequalities satisfied by both the relative entropy and the Fisher Information of $k$ particles. We also show with a simple Gaussian example that our decay rate is optimal.
\end{abstract}

\section{Introduction}
Fisher information sits at the crossroads of inference, analysis and kinetic theory: it controls fluctuations of densities, underpins de Bruijn–type relations connecting entropy production to information and refines classical information-theoretic inequalities such as Stam’s inequality, see for example the recent survey \cite{Vil25}. In diffusive mean-field (McKean–Vlasov) models, the Fokker–Planck structure makes the entropy–information dissipation explicit and places the dynamics within the Wasserstein gradient-flow framework, where log-Sobolev and transport inequalities \cite{OV00} turn information dissipation into quantitative stability \cite{CMV03}. Superadditivity/tensorization of Fisher information  properties \cite{Car91} further propagate regularity across coordinates, a feature that resonates with the chaoticity of many-particle systems. This perspective dovetails with the propagation-of-chaos program—from the foundational works and modern reviews to recent entropy-based, quantitative approaches—and with singular kinetic examples (e.g. 2D vortices) where uniform Fisher-information bounds yield entropic (and thus strong) chaos as in \cite{FHM14}. Very recently, Fisher information has been central to prove new striking results for Landau \cite{GS25} and Boltzmann \cite{ISV25} equations: monotonicity along the flow in the physical Coulombian case. In this paper we develop this line and prove propagation of chaos in Fisher information for a broad class of mean-field diffusions, leveraging  entropy–information flows across levels of the BBGKY hierarchy.

We consider here a system of $N$ diffusion processes with mean-field interaction, evolving in $\mathbb{R}^d$ (or the d-dimensional torus $\mathbb{T}^d$) according to the system of stochastic differential equations
\begin{equation*}
    {\rm d} X^{i,N}_t = F(X^{i,N}_t) {\rm d} t + \frac{1}{N}\sum_{j=1}^N \Gamma(X^{i,N}_t-X^{j,N}_t){\rm d}t+ \sqrt{2}{\rm d}B^i_t,
\end{equation*}
where $(B^i)_{i= 1\ldots N}$ is a family of independent standard Brownian motions and $\Gamma$ and $F$ are supposed smooth, with bounded derivatives, with $\Gamma$ moreover supposed bounded. It is well known that, for $k$ fixed and $N$ large, the distribution $\mu^{k,N}$ of $(X^{1,N}_t,\ldots X^{k,N}_t)$ is close to the distribution $\bar  \mu_{t}^{\otimes k}$ of $k$ independent processes $(\bar X^{1}_t,\ldots \bar X^{k}_t)$, with $\bar X^i_t$ solutions to
\begin{equation*}
    {\rm d} \bar X^{i}_t = F(\bar X^i_t){\rm d} t+  \Gamma*\bar \mu_t (\bar X^i_t){\rm d}t+ \sqrt{2}{\rm d}B_t,
\end{equation*}
with $B$ a standard Brownian motion. This convergence, known as {\sl propagation of chaos}, has been studied extensively since the second half of the 20th century due to its deep ties with the kinetic theory of gases. Several methods have been devised to prove propagation of chaos, starting with compactness arguments \cite{Snitzman}, \cite{Meleard} and coupling arguments \cite{Snitzman_topics}, \cite{McKean}, \cite{Eberle}. Entropy related methods have garnered a lot of attraction over the last decade, notably due to their ability to handle singular interactions, a notable case being the 2D Vortex model studied in \cite{Jabin_Wang}. It usually consists of estimating $H_{t}^{N} \coloneqq H\left(\mu_{t}^{N,N}|\mu_{t}^{\otimes N}\right)$ by a constant uniform in N and then relying on the subadditivity of the relative entropy to obtain $H_{t}^{k} \coloneqq H\left(\mu_{t}^{k,N}|\mu_{t}^{\otimes k}\right)=\mathcal{O}\left(\frac{k}{N}\right)$. It turns out that the subadditivity fails to yield the optimal rates of convergence. For a more detailed review, see \cite{Chaintron_1}, \cite{Chaintron_2}.

More recently, Lacker \cite{lacker} was able to obtain optimal rates of convergence of order $H_{t}^{k}=\mathcal{O} \left(\frac{k^{2}}{N^{2}} \right)$ for propagation of chaos under suitable assumptions on the interaction (that are satisfied if $\Gamma$ is bounded or Lipschitz for instance), by working directly with $H_{t}^{k}$ instead of $H_{t}^{N}$. The novelty of this approach is the use of the BBGKY hierarchy to obtain a system of differential inequalities of the form
\[ \frac{d}{dt} H_{t}^{k} \leq C \frac{k^{2}}{N^{2}}+C  k \left(H_{t}^{k+1}-H_{t}^{k} \right). \]
This approach has seen several improvements since this seminal work. Uniform in time propagation of chaos in entropy has been obtained in \cite{lacker_uniforme}, by relying on log-Sobolev inequalities to gain an additional $-c H_{t}^{k}$ on the right hand side of the system, and thus needing additional assumptions. The non exchangeable case has also been considered in \cite{lacker_graphe} where sharp convergence rates are obtained for systems of particles with weighted interaction (for instance, when the particles that interact are given by a deterministic matrix) toward $n$ independent processes (the so-called independent projection). See also \cite{Hess_Childs} for the $\chi^{2}$ divergence instead of the relative entropy, where sharp convergence rates are also obtained for higher order corrections of the mean-field limit.

A recent breakthrough for singular interactions has been made in \cite{Vortex}, that was done by obtaining a more general system of differential inequalities of the form
\[\frac{d}{dt} H_{t}^{k} \leq -c_{1} I_{t}^{k}+c_{2} I_{t}^{k+1}
{\bf 1}_{k<N}+C\frac{k^{2}}{N^{2}}+ C \ k \left(H_{t}^{k+1}-H_{t}^{k} \right){\bf 1}_{k<N}, \]
where $c_{1}>c_{2}$ and $\displaystyle I_{t}^{k}\coloneqq \sum_{i=1}^{k} \int \mu_{t}^{k,N} \norm{\nabla_{v_{i}} \log \frac{\mu_{t}^{k,N}}{\mu_{t}^{\otimes k}}}^{2} $. By using a change of variable, it is then possible to obtain a system of the form
\[\frac{d}{dt} z_{t}^{k} \leq C \frac{k^{2}}{N^{2}}+C  k \left(z_{t}^{k+1}-z_{t}^{k} \right){\bf 1}_{k<N}.\]

This approach has been used to treat the case of non constant diffusion coefficient by the authors in \cite{article_diffusion}, by using a new (in the context of propagation of chaos) decomposition of the Fisher Information, and it seems that a better understanding of $I_{t}^{k}$ could yield new results. A striking point in the change of variable in \cite{Vortex} is that it does not rely on any property of the Fisher information apart its positivity. In fact, little is known about $I_{t}^{k}$ and propagation of chaos in Fisher Information is very difficult to prove. To the best of our knowledge, very few papers consider propagation of chaos in Fisher information. In \cite{Hauray_Mischler}, it is proved that Fisher Information propagation of chaos is stronger than its entropic counterpart, but up to our reading no criterion to prove the Fisher propagation of chaos is given. For a quantitative version it can be established through a logarithmic Sobolev inequality for $\mu_t$ to control the entropy, and the $W_2$-distance. Note also that \cite{chewi} proves the static propagation of chaos (without sharp rates) in Fisher information, which is far simpler as  the measures are in this case explicit. The present article fills that gap in the literature. It also opens further developments that we will consider in the near future: singular interaction kernels (recall that it has been established for entropy with sharp rate in \cite{Vortex}, the uniform in time case, non constant diffusion... We hope also that it can lead to an
alternative approach to the monotonicity of the Fisher information in the Landau or Boltzmann equation, using this monotonicity for the particle system and then the propagation of chaos in Fisher information. Note finally that propagation of chaos brings new information wrt the one for entropy: it is well known that Fisher information controls difference of the scores $\log \mu_t^{k,N} -\log \mu_t^{\otimes k}$ (which may also be seen as an $H^1$ type control) and is far more sensible on the oscillation of the potentials -- consider on the torus $\mu_n= (1+\sin(nx)/n^{1/2})dx$ for which $H(\mu_n|dx)\to0$ whereas $I(\mu_n|dx)\to\infty$--.

One of our crucial tool is a new Lemma for estimating the Fisher Information between two diffusion processes with smooth drifts and constant noise and adapting Lacker's method (in the case of smooth interaction) we obtain a system of differential inequalities of the form
\[\frac{d}{dt}I_{t}^{k}  \leq C \frac{k^{2}}{N^{2}}+Ck\left(H_{t}^{k+1}-H_{t}^{k} \right){\bf 1}_{k<N}+Ck\left(I_{t}^{k+1}-I_{t}^{k} \right){\bf 1}_{k<N}+C \ I_{t}^{k},\]
which allows us to obtain the bound $I_{t}^{k}=\mathcal{O} \left(\frac{k^{2}}{N^{2}} \right)$.
We then prove from this inequality and its counterpart in entropy that the optimal rate of convergence is of order $\frac{k^{2}}{N^{2}}$, which means that our decay in $N$ is optimal.
The proof also relies on fine estimates on $\mu_{t}$ and its spacial derivatives by using Hamilton-Jacobi-Bellman equations. 

The plan of the paper is the following. Section \ref{MainRes} presents our main results concerning the propagation of chaos in Fisher information, as well as the crucial Lemma evaluating the derivative of the Fisher information between two solutions of inhomogeneous in time Fokker Planck equations with different drifts and constant diffusion, whose proof is given in Section \ref{ProofLem}. The proof of the main result using BBGKY hierarchy in Fisher information is done in Section \ref{ProofThm}. We present in Section \ref{Gauss} an explicit example, i.e. mean field Ornstein-Uhlenbeck process, where the Gaussian structure enables us to derive the exact rate of propagation of chaos in Fisher information, confirming the sharpness of the rate obtained in our main result. Finally Section \ref{Reg} includes all the fine regularity estimates necessary to justify the derivations of Section \ref{ProofLem} and \ref{ProofThm}.

\section{Main results}
\label{MainRes}
Let us first recall the definition of the relative entropy and Fisher information : for any probability measures $\mu$ and $\nu$ such that $\mu$ admits $f$ as density with respect to $\nu$, we define
\begin{equation*}
    H(\mu|\nu) = \int \nu f \log f,
\end{equation*}
and
\begin{equation*}
    I(\mu|\nu) = \int \nu f \Vert \nabla \log f\Vert^2 = \int \nu \frac{\Vert \nabla f\Vert^2}{f}.
\end{equation*}
We will make the following hypotheses for the initial conditions throughout the paper.

\begin{hyp}
\label{hyp-princ}
\begin{itemize}
    \item Torus Case. We suppose that $\bar \mu_0$ is $C^2$, positive and that  $\nabla^2 \log \bar\mu_0$ is bounded.
    \item $\mathbb{R}^d$ Case. We suppose initial Gaussian bounds. More precisely that there exist constants $c$, $C$, $c'$, $C'$ such that
    \begin{equation}\label{hyp:gaussian tail I}
    C^{-1}e^{-c^{-1}\norm{v}^2}\leq \mu^{N,N}_0(v)\leq Ce^{-c\norm{v}^2},\quad  C'^{-1}e^{-c'^{-1}\norm{v}^2}\leq \bar \mu_0(v)\leq C'e^{-c'\norm{v}^2},
    \end{equation}
    and that, for $j=1,2,3$,
    \begin{equation}
    \norm{\nabla^j \log \mu^{N,N}_0}^2 \leq C\left(1+ \left|\log \mu^{N,N}_{0}\right|^j\right), \quad \norm{\nabla^j \log \bar\mu_0}^2 \leq C'\left(1+ \left|\log \bar \mu_{0}\right|^j\right).
    \end{equation}
\end{itemize}
\end{hyp}
Remark that in the $\mathbb{R}^d$ case, the assumption is indeed verified for $\mu_0^{N,N}$ a Gaussian measure.

We may now state the main result of our paper.

\begin{theorem}\label{th:main}
Fix a $T>0$ and suppose that there exists a constant $C$ such that
\begin{equation*}
    H(\mu^{k,N}_0|\bar \mu_0^{\otimes k}) \leq C\frac{k^2}{N^2} \quad  \text{and} \quad I(\mu^{k,N}_0|\bar \mu_0^{\otimes k}) \leq C\frac{k^2}{N^2}.
\end{equation*}
Let distinguish two set of assumptions:
\begin{itemize}
    \item Torus Case. Suppose that $F$ and $\Gamma$ are smooth.
    \item $\mathbb{R}^d$ case. Suppose that $\Gamma$ is smooth bounded, with bounded derivatives, that $F$ is smooth with bounded derivatives and assume moreover that
    \begin{equation*}
    \sup_{t\in [0,T]}\Vert \nabla^2 \log \bar \mu_t\Vert <\infty.
\end{equation*}
\end{itemize}
We suppose here that Hypothesis \ref{hyp-princ} is effective.
Then there exists a constant $M_T$ 
such that
\begin{equation*}
    H(\mu^{k,N}_t|\bar \mu_t^{\otimes k})+I(\mu^{k,N}_t|\bar \mu_t^{\otimes k}) \leq M_T \frac{k^2}{N^2}.
\end{equation*}
\end{theorem}

\begin{rem}
\begin{itemize}
    \item We suppose for simplicity that all derivatives are bounded, but as can be seen in the proof, derivatives up to four may be sufficient, however leading to refined regularity estimates.
    \item The hypothesis $\sup_{t\in [0,T]}\Vert \nabla^2 \log \bar \mu_t\Vert <\infty$ is in fact trivially verified in the torus setting, if assumed at time $t=0$, for example via Feynman-Kac formula due to the ellipticity assumption and regularity of our coefficients. In the $\mathbb{R}^d$ case, we will rely on recent results on Hamilton-Jacobi-Bellman equation derived in \cite{bdd_hess}. It will be the purpose of Proposition \ref{LemDev2}. We prefer to state this as an assumption in our main Theorem as it will be a crucial step to extend our results on models with less regularity (vortex 2D, ...) or uniformly in time that will be considered in future works.
\end{itemize}
\end{rem}

The proof of our main result relies on the well known BBGKY hierarchy (which was already the crucial ingredient to obtain the optimal rates in \cite{lacker,lacker_uniforme,lacker_graphe,Vortex}) : $\mu^{k,N}_t$ is a solution to the PDE 
\begin{equation}\label{eq:mukNt}
    \partial_t \mu^{k,N}_t = \sum_{i=1}^k \Delta_{v_i} \mu^{k,N}_t -\sum_{i=1}^k \nabla_{v_i}\cdot \left( b^{i,k}_{1,t} \mu^{k,N}_t\right), 
\end{equation}
with
\begin{equation}\label{eq:def bi1}
b^{i,k}_{1,t}(v_1,\ldots,v_k) = F(v_i) + \frac{1}{N}\sum_{j=1}^k \Gamma(v_i-v_j) + \frac{N-k}{N}\int \Gamma(v_i-v_{k+1}) \mu^{k+1|k}_t(v_{k+1}|v_1,\ldots,v_k),
\end{equation}
where $\mu^{k+1|k}_t(v_{k+1}|v_1,\ldots,v_k)$ denotes the distribution of the $(k+1)$-th particle $X^{k+1,N}_t$ conditioned to $(X^{1,N}_t=v_1,\ldots, X^{k,N}_t=v_k)$.
To prove propagation of chaos, the idea is to compare $\mu^{k,N}_t$ to $\bar\mu^{\otimes k}_t$, which is a solution to the following PDE:
\begin{equation}\label{eq:mubarkt}
   \partial_t \bar\mu^{\otimes k}_t = \sum_{i=1}^k\Delta_{v_i} \bar\mu^{\otimes k}_t - \sum_{i=1}^k\nabla_{v_i}\cdot \left(b^{i}_{2,t} \bar\mu^{\otimes k}_t\right),  
\end{equation}
with
\begin{equation}\label{def:b2}
    b^{i}_{2,t}(v_i) = F(v_i) + \Gamma*\bar\mu_t (v_i).
\end{equation}
The main step of our proof is the following Lemma, 
that controls the evolution of the Fisher information of
two generic probability distributions $m_1$ and $m_2$ that are solutions to PDEs similar respectively to \eqref{eq:mukNt} and \eqref{eq:mubarkt}.

\begin{lemma}\label{lem: I(m1|m2)}
Suppose that $m_{1,t}$ and $m_{2,t}$ are respectively a solution to, for $u=1,2$,
\begin{equation}\label{eq: evolution m_u}
\partial_{t} m_{u,t}=\sum_{i=1}^k \Delta_{v_i} m_{u,t}-\sum_{i=1}^k\nabla_{v_i} \cdot \left(b^i_{u,t} m_{u,t} \right),
\end{equation}
that $m_{1,t}$ and $m_{2,t}$ are of $C^3$ regularity and that, for any $T>$, there exists constants $\kappa_{i,j}$, $\kappa'_{i,0}$, $i=1,2$, $j=0,1,2,3$ such that, for any $t\in [0,T]$,
\begin{equation}\label{hyp:gauss unif time}
  \kappa_{i,0}^{-1} -(\kappa_{i,0}')^{-1} \norm{v}^2 \leq \log m_{i,t}(v) \leq \kappa_{i,0} -\kappa_{i,0}' \norm{v}^2,  
\end{equation}
and
\begin{equation}\label{hyp:gassian tail II unif time}
    \norm{\nabla^j  m_{i,t}(v)} \leq \kappa_{i,j}\left(1+ \left|v\right|^j\right)m_{i,t}(v).
\end{equation}
 Suppose that $b^i_{u,t}$ is smooth, with bounded derivatives for any $u=1,2$ and $i=1 \dots k$ and moreover that $b^i_{1,t}-b^i_{2,t}$ is bounded for $i=1 \dots k$.  Denote $f_t=m_{1,t}/m_{2,t}$. We suppose that for any $T>0$,
\begin{equation*}
   \sup_{t\in [0,T]} \int  m_{2,t}\frac{\Vert \nabla^2 f_t\Vert^2}{f_t}<\infty, \quad \sup_{t\in [0,T]}\int  m_{2,t}\frac{\Vert\nabla^3 f_t\Vert^2}{f_t}<\infty, \quad \sup_{t\in [0,T]} \int m_{2,t}\frac{\Vert \nabla f_t\Vert^4}{f_t^3}<\infty,
\end{equation*}
\begin{equation*}
   \sup_{t\in [0,T]} \int  m_{1,t} \Vert \nabla^2 \log m_{2,t}\Vert^2 <\infty, \quad \sup_{t\in [0,T]} \int  m_{1,t} \Vert \nabla \log m_{2,t}\Vert^4 <\infty,
\end{equation*}
\begin{equation*}
   \sup_{t\in [0,T]  } \int m_{1,t}  \Vert b_{2,t} \Vert^4 <\infty.
\end{equation*}
Then, for any $\varepsilon>0$ and any $T>0$ there exists a constant $C$,  depending on $\varepsilon$,
such that for all $t\in [0,T]$,
\begin{align*}
\frac{d}{dt} I(m_{1,t}|m_{2,t}) &\leq -(2-\varepsilon)\sum_{i,j=1}^k  \int m_{2,t} f_t\left\Vert \nabla^2_{v_i,v_j} \log f\right\Vert^2  \\
&\quad +4\sum_{i,j=1}^k\int m_{2,t} f_t \ \nabla_{v_i} \log f_t \cdot \nabla^2_{v_i,v_j} \log m_{2,t} \nabla_{v_j} \log f_t
\\&
\quad -2 \sum_{i,j=1}^{k} \int m_{2,t} f_t \ \nabla_{v_{i}} \log f_t \cdot \nabla_{v_{i}} b_{2,t}^{j} \ \nabla_{v_{j}} \log f_t\\
&
\quad +2  \sum_{i,j=1}^{k} \int m_{2,t} \frac{\nabla_{v_{i}} f_t}{f_t} \cdot \nabla^2_{v_{i},v_{j}} \log m_{2,t}  \left(b_{2,t}^{j}-b_{1,t}^{j}\right) \ f_t 
\\ &\quad + C\sum_{i,j=1}^{k} \int m_{1,t} \norm{\nabla_{v_{i}} \left(b_{2,t}^{j}-b_{1,t}^{j} \right)}^{2}.
\end{align*}

\end{lemma}

\begin{rem}
\begin{itemize}
\item We have stated our result for two generic solutions of \ref{eq: evolution m_u}, but since we will be using Lemma \ref{lem: I(m1|m2)} with $m_{2,t}=\mu_{t}^{\otimes k}$, there will be additional cancellations. The main bound we use in Section \ref{th:main} is
\begin{align*}
\frac{d}{dt} I_{t}^{k} \leq C I_{t}^{k}+\sum_{i=1}^{k} \int \mu_{t}^{k,N} \ \norm{b_{2,t}^{i,k}-b_{1,t}^{i,k}}^{2}+\sum_{i,j=1}^{k} \int \mu_{t}^{k,N} \ \norm{\nabla_{v_{i}} \left(b_{2,t}^{j,k}-b_{1,t}^{j,k} \right)}^{2}
\end{align*}
\item For simplicity we have stated this result supposing that $b^i_{1,t}$ is smooth and bounded. So, while $m_2$ may be directly replaced by $\bar \mu^{\otimes k}_t$, it is not the case for $m_{1,t}$ and $\mu^{k,N}_t$ in the $\mathbb{R}^d$ setting: $b^{i,k}_{1,t}-b^{i,k}_{2,t}$ is indeed bounded but in \eqref{eq:def bi1} $b^{i,k}_{1,t}$ does not have its derivatives bounded since it involves the conditional distribution $\mu^{k+1|k}_t$. In Section~\ref{Reg} we will show that the computations made to prove this Lemma are still valid when $b^i_{1,t}$ is replaced by $b^{i,k}_{1,t}$ and that the other hypotheses of Lemma~\ref{lem: I(m1|m2)} are satisfied in the case we are interested in.
\end{itemize}
\end{rem}
Let us conclude this Section by a Proposition stating that, in the gradient case, our assumption on the Hessian of the logarithm of the nonlinear flow is satisfied under our assumptions. This Proposition is obtained by an application of a result of \cite{bdd_hess}. The details of the proof are given in Section~\ref{Reg}

\begin{proposition}
\label{LemDev2}
Suppose that there exists two functions $V, W$ such that $F=\nabla V$ and $\Gamma=\nabla W$ and suppose moreover that the derivatives of $F$ and $\Gamma$ are bounded, and that $\Gamma$ is bounded. Suppose that $\nabla^{2} \log \bar \mu_{0}$ is bounded. Then, for all $T \geq 0$, $\nabla^{2} \log \bar\mu_{t}$ is uniformly bounded on $[0,T]$.
\end{proposition}
Note that here we do not have assumptions on the ergodicity of our processes as only growth assumptions are made on the drift and interaction terms and therefore no uniform in time bound can be obtained in Prop. \ref{LemDev2}. Note however that in Section \ref{Gauss}, when $a$ is positive and $a+2b$, as densities are gaussian, it is easy to verify that $\nabla^{2} \log \bar\mu_{t}$ is uniformly bounded in time and space. It will be a crucial result to prove for uniform in time propagation in Fisher information. 

\section{Proof of Lemma~\ref{lem: I(m1|m2)}}
\label{ProofLem}

Recall that we have defined $f_t=m_{1,t}/m_{2,t}$. We start by decomposing the terms involved in the dynamics of $f_t$. For simplicity of notations we drop the dependency in time, and write $f$, $m_1$, $m_2$ instead of $f_t$, $m_{1,t}$, $m_{2,t}$... Recalling \eqref{eq: evolution m_u} we obtain
\begin{align*}
\partial_{t} f & = \frac{\partial_{t} m_{1}}{m_{2}}-\frac{\partial_{t} m_{2}}{m_{2}^{2}} \ m_{1}
\\ 
&=\underbrace{\sum_{i=1}^{k}-\frac{\nabla_{v_{i}} \cdot \left(b_{1}^{i} m_{1} \right)}{m_{2}}+\nabla_{v_{i}} \cdot \left(b_{2}^{i} m_{2} \right) \ \frac{m_{1}}{m_{2}^{2}}}_{=: \mathcal{L}_{1}  f}
+\underbrace{\sum_{i=1}^{k} \ f \ \left( \frac{\Delta_{v_{i}} m_{1}}{m_{2}}-\frac{\Delta_{v_{i}} m_{2}}{m_{2}^{2}} \ m_{1} \right)}_{=: \mathcal{L}_{2} f}.
\end{align*}
On one hand, focusing on the operator $\mathcal{L}_1$, we have
\begin{align*}
\mathcal{L}_{1} f& =\sum_{i=1}^{k} -\nabla_{v_{i}} \cdot b_{1}^{i} \ f-b_{1}^{i} \cdot \frac{\nabla_{v_{i}} m_{1}}{m_{2}}+\nabla_{v_{i}} \cdot b_{2}^{i} \ f+   b_{2}^{i} \cdot \nabla_{v_{i}} \log m_{2}\ f
\\ 
&=\sum_{i=1}^{k} \nabla_{v_{i}} \cdot \left(b_{2}^{i}-b_{1}^{i} \right) \ f-b_{1}^{i} \cdot \left(\frac{\nabla_{v_{i}} m_{1}}{m_{2}}-\nabla_{v_{i}} m_{2} \frac{m_{1}}{m_{2}^{2}} \right)+\left(b_{2}^{i}-b_{1}^{i} \right) \cdot \nabla_{v_{i}} \log m_{2} \ f
\\
& = \sum_{i=1}^{k} \nabla_{v_{i}} \cdot \left( b_{2}^{i}-b_{1}^{i} \right) \ f-b_{1}^{i} \cdot \nabla_{v_{i}} f+\left(b_{2}^{i}-b_{1}^{i} \right)\cdot \nabla_{v_{i}} \log m_{2} \ f,
\end{align*}
and thus we obtain
\begin{equation}\label{eq: cL1}
    \mathcal{L}_{1} = \sum_{i=1}^{k} \nabla_{v_{i}} \cdot \left(\left( b_{2}^{i}-b_{1}^{i} \right) \ f\right)-  b^i_2\cdot \nabla_{v_i}f +\left(b_{2}^{i}-b_{1}^{i} \right)\cdot \nabla_{v_{i}} \log m_{2} \ f.
\end{equation}
On the other hand, by standard computations,
\begin{align*}
\nabla_{v_{i}} \frac{m_{1}}{m_{2}}=\frac{\nabla_{v_{i}} m_{1}}{m_{2}}-\nabla_{v_{i}} m_{2} \frac{m_{1}}{m_{2}^{2}},
\end{align*}
and
\begin{align*}
\Delta_{v_{i}} \frac{m_{1}}{m_{2}}&=\nabla_{v_{i}} \cdot \nabla_{v_{i}} \frac{m_{1}}{m_{2}}
\\ &
=\nabla_{v_{i}} \cdot \left(\frac{\nabla_{v_{i}} m_{1}}{m_{2}}-\nabla_{v_{i}} m_{2} \frac{m_{1}}{m_{2}^{2}}\right)
\\ &
=\frac{\Delta_{v_{i}} m_{1}}{m_{2}}-\nabla_{v_{i}} m_{1} \cdot \frac{\nabla_{v_{i}} m_{2}}{m_{2}^{2}}-\Delta_{v_{i}} m_{2} \ \frac{m_{1}}{m_{2}^{2}}-\nabla_{v_{i}} m_{1} \cdot \frac{\nabla_{v_{i}} m_{2}}{m_{2}^{2}}+2 \ \nabla_{v_{i}} m_{2} \cdot \nabla_{v_{i}} m_{2} \ \frac{m_{1}}{m_{2}^{3}},
\end{align*}
which leads to
\begin{align*}
\sum_{i=1}^{k} \Delta_{v_{i}} \ f& =\mathcal{L}_{2} f+ 2 \ \sum_{i=1}^{k} \nabla_{v_{i}} m_{2} \cdot \left(\nabla_{v_{i}} m_{2} \ \frac{m_{1}}{m_{2}^{3}}-\frac{\nabla_{v_{i}} m_{1}}{m_{2}^{2}} \right)
\\ &
=\mathcal{L}_{2} f+2 \ \sum_{i=1}^{k} \nabla_{v_{i}} \log m_{2} \cdot \left(\nabla_{v_{i}} m_{2} \frac{m_{1}}{m_{2}^{2}}-\sum_{i=1}^{k} \frac{\nabla_{v_{i}} m_{1}}{m_{2}} \right)
\\ &
=\mathcal{L}_{2} f -2 \ \sum_{i=1}^{k} \nabla_{v_{i}} \log m_{2} \cdot \nabla_{v_{i}} f.
\end{align*}
Therefore:
\begin{equation}\label{eq: cL2}
\mathcal{L}_{2} f=\sum_{i=1}^{k} \Delta_{v_{i}} f+ 2 \ \n_{v_{i}} \log m_{2} \cdot \n_{v_{i}} f.
\end{equation}

\medskip

We can now focus on the evolution of the Fisher Information.
Remark that if we define
\[
I^R(m_1|m_2) = \sum_{i=1}^k  \int K_R m_2\frac{\Vert \nabla_{v_i} f\Vert^2}{f},
\]
where $K_R(v)$ has value $1$ if $|v|\leq R$ and value $0$ id $|v|\geq R$, then we can differentiate with respect to time inside the integral in $I^R(m_1|m_2)$, and write
\[
I^R(m_{1,t}|m_{2,t})-I^R(m_{1,0}|m_{2,0}) = \int_0^t \frac{d}{ds}I^R(m_{1,s}|m_{2,s}) {\rm d} s.
\]
Remark moreover that one can apply the monotone convergence Theorem on the left-hand side. Moreover one can show via dominated convergence Theorem that each term appearing in $\frac{d}{ds}I^R(m_{1,s}|m_{2,s})$ admits a limit (the domination, uniform in $R$ and $s\in [0,T]$, is given by the bound of each term provided at the end of the proof), and defining
\[
\frac{d}{ds} I(m_{1,s}|m_{2,s}) = \lim_{R\rightarrow \infty} \frac{d}{ds} I^R(m_{1,s}|m_{2,s}),
\]
applying again the dominated convergence Theorem we get
\[
I(m_{1,t}|m_{2,t})-I(m_{1,0}|m_{2,0}) = \int_0^t \frac{d}{ds}I(m_{1,s}|m_{2,s}) {\rm d} s,
\]
with $ \frac{d}{ds}I(m_{1,s}|m_{2,s})  \in L^\infty([0,T])$. More precisely,
\begin{align}
&\frac{d}{dt} I(m_{1}|m_{2}) \nonumber\\
&= \sum_{i=1}^{k} \int \partial_{t} m_{2} \ \frac{\norm{\n_{v_{i}} f}^{2}}{f}+2 \ \int \frac{m_{2}}{f} \ \n_{v_{i}} \left( \partial_{t} f \right) \cdot \n_{v_{i}} f-\int m_{2} \ \norm{\n_{v_{i}} f}^{2} \ \frac{\partial_{t} f}{f^{2}}
\nonumber\\ &
=\underbrace{-\sum_{i=1}^{k} \int \sum_{j=1}^k\nabla_{v_{j}} \cdot \left(b^j_{2} m_{2} \right) \frac{\norm{\nabla_{v_{i}} f}^{2}}{f}+2 \ \int \frac{m_{2}}{f} \  \nabla_{v_{i}} \left(\mathcal{L}_{1} f\right) \cdot \nabla_{v_{i}} f-\int m_{2} \ \norm{\nabla_{v_{i}} f}^{2} \ 
\frac{\mathcal{L}_{1} f}{f^{2}}}_{=: I_{1}}
\nonumber\\ &
\quad +\underbrace{\sum_{i=1}^{k} \int \sum_{j=1}^k\Delta_{v_{j}} m_{2} \frac{\norm{\nabla_{v_{i}} f}^{2}}{f}+2 \ \int \frac{m_{2}}{f} \  \nabla_{v_{i}} \left(\mathcal{L}_{2} f\right) \cdot \nabla_{v_{i}} f-\int m_{2} \ \norm{\nabla_{v_{i}} f}^{2} \ 
\frac{\mathcal{L}_{2} f}{f^{2}}}_{=: I_{2}}.\label{eq: dyn I(m1,m2)}
\end{align}
In this decomposition $I_2$ corresponds to the effect of the diffusion, while $I_1$ corresponds to the effects of the drift. We will treat these two effects separately in the two following subsections, the proof of Lemma~\ref{lem: I(m1|m2)} being a direct consequence of \eqref{eq: expression I2} and \eqref{eq: bound I1}.

\subsection{Dissipation of Fisher information: the diffusion term}

The aim of this subsection is to obtain an adequate expression of term $I_2$. This calculation is inspired by \cite[Part 2, Chap 4, p.280]{Vil98} or \cite[Lem. 4.2]{Tos99}, which focus on the flat Fisher information. Recalling \eqref{eq: dyn I(m1,m2)} and \eqref{eq: cL2} we obtain
\begin{align*}
I_{2} =\sum_{i,j=1}^{k}& \underbrace{\int \Delta_{v_{i}} m_{2} \frac{\norm{\nabla_{v_{j}} f}^{2}}{f}}_{=: A^{i,j}}+ \underbrace{2 \int \frac{m_{2}}{f} \  \nabla_{v_{i}} \left(\Delta_{v_j} f\right) \cdot \nabla_{v_{i}} f}_{=:B^{i,j}}\\
&+ \underbrace{4\int \frac{m_{2}}{f} \  \nabla_{v_{j}} \left(\nabla_{v_i} \log m_2 \cdot \nabla_{v_i}f\right) \cdot \nabla_{v_{j}} f}_{=:C^{i,j}}\
\underbrace{-\int m_{2} \ \norm{\nabla_{v_{j}} f}^{2} \ 
\frac{\Delta_{v_i} f}{f^{2}}}_{=:D^{i,j}}\\
&\underbrace{-2\int m_{2} f\ \left\Vert\nabla_{v_{j}} \log f\right\Vert^{2} \ 
\nabla_{v_i} \log m_2 \cdot \nabla_{v_i}\log f}_{=:E^{i,j}}.
\end{align*}
Relying on integration by parts, we obtain
\begin{align*}
    B^{i,j} &= 2 \sum_{\alpha,\beta=1}^d \int \frac{m_2}{f} \frac{\partial^3 f}{\partial^2 {v_{i,\alpha}}\partial {v_{j,\beta}}}\frac{\partial f}{\partial {v_{j,\beta}}}\\
    &= -2\sum_{\alpha,\beta=1}^d \int \frac{\partial^2 f}{\partial {v_{i,\alpha}}\partial {v_{j,\beta}}}\frac{\partial}{\partial {v_{i,\alpha}}}\left(\frac{m_2}{f}\frac{\partial f}{\partial {v_{j,\beta}}}\right)\\
    &= -2 \int m_2 \nabla_{v_i} \log m_2 \cdot \nabla^2_{v_i,v_j} f\ \nabla_{v_j} \log f - 2 \int \frac{m_2}{f}\left\Vert \nabla^2_{v_i,v_j} f\right\Vert^2 \\
    &\quad + 2 \int m_2 \nabla_{v_i} \log f \cdot \nabla^2_{v_i,v_j} f\ \nabla_{v_j} \log f,
\end{align*}
where $\nabla^2_{v_i,v_j}f$ is the matrix $\displaystyle \left(\frac{\partial^2 f}{\partial v_{i,\alpha}\partial v_{j,\beta}} \right)_{1\leq \alpha,\beta\leq d}$, and
\begin{align*}
    D^{i,j} &= -  \sum_{\alpha,\beta=1}^d \int \frac{m_2}{f^2} \frac{\partial^2 f}{\partial v_{i,\alpha}^2}\left(\frac{\partial f}{\partial v_{j,\beta}}\right)^2 \\
    &=\sum_{\alpha,\beta=1}^d \int  \frac{\partial f}{\partial v_{i,\alpha}}\frac{\partial}{\partial v_{i,\alpha}}\left(\frac{m_2}{f^2}\left(\frac{\partial f}{\partial v_{j,\alpha}}\right)^2\right)\\
    &= \int m_2 f \left\Vert \nabla_{v_j}\log f \right\Vert^2 \nabla_{v_i} \log m_2 \cdot \nabla_{v_i}  \log f + 2\int m_2 \nabla_{v_i} \log f \cdot \nabla^2_{v_i,v_j} f\  \nabla_{v_j} \log f\\
    &\quad -2\int m_2 f \left\Vert \nabla_{v_i}\log f \right\Vert^2 \left\Vert \nabla_{v_j}\log f \right\Vert^2.
\end{align*}
These two computations lead to
\begin{align*}
    B^{i,j}+D^{i,j} &=- 2 \int \frac{m_2}{f}\left\Vert \nabla^2_{v_i,v_j} f\right\Vert^2+ 4\int m_2 \nabla_{v_i} \log f \cdot \nabla^2_{v_i,v_j} f\  \nabla_{v_j} \log f\\
    &\quad -2\int m_2 f \left\Vert \nabla_{v_i}\log f \right\Vert^2 \left\Vert \nabla_{v_j}\log f \right\Vert^2\\
    &\quad -2 \int m_2 \nabla_{v_i} \log m_2 \cdot \nabla^2_{v_i,v_j} f\ \nabla_{v_j} \log f+ \int m_2 f \left\Vert \nabla_{v_j}\log f \right\Vert^2 \nabla_{v_i} \log m_2 \cdot \nabla_{v_i}  \log f \\
    & = -2\int m_2 f\left\Vert \nabla^2_{v_i,v_j} \log f\right\Vert^2 \underbrace{-2 \int m_2 \nabla_{v_i} \log m_2 \cdot \nabla^2_{v_i,v_j} f\ \nabla_{v_j} \log f}_{=:F^{i,j}}-\frac{E^{i,j}}{2},
\end{align*}
and thus, remarking that
\begin{align*}
   \frac{\left\Vert \nabla^2_{v_i,v_j} f\right\Vert^2}{f^2} -& 2 \frac{\nabla_{v_i} \log f \cdot \nabla^2_{v_i,v_j} f\  \nabla_{v_j} \log f}{f}+ \left\Vert \nabla_{v_i}\log f \right\Vert^2 \left\Vert \nabla_{v_j}\log f \right\Vert^2\\
   & = \sum_{\alpha,\beta=1}^d \left(\frac{1}{f}\frac{\partial^2 f}{\partial v_{i,\alpha}\partial v_{j,\beta}}\right)^2- \frac{2}{f}\frac{\partial^2 f}{\partial v_{i,\alpha}\partial v_{j,\beta}}\frac{1}{f^2}\frac{\partial f}{\partial v_{i,\alpha}}\frac{\partial f}{\partial v_{j,\beta}}+\left(\frac{1}{f^2}\frac{\partial f}{\partial v_{i,\alpha}}\frac{\partial f}{\partial v_{j,\beta}}\right)^2\\
   &=\sum_{\alpha,\beta=1}^d \left( \frac{1}{f}\frac{\partial^2 f}{\partial v_{i,\alpha}\partial v_{j,\beta}}-\frac{1}{f^2}\frac{\partial f}{\partial v_{i,\alpha}}\frac{\partial f}{\partial v_{j,\beta}}\right)^2\\
   &=\left\Vert \nabla^2_{v_i,v_j} \log f\right\Vert^2,
\end{align*}
we deduce
\begin{align*}
    I_2 =\sum_{i,j=1}^k -2\int m_2 f\left\Vert \nabla^2_{v_i,v_j} \log f\right\Vert^2 + A^{i,j}+C^{i,j}+\frac{E^{i,j}}{2}+F^{i,j}.
\end{align*}
Now, relying on an integral by part,
\begin{align*}
    F^{i,j} &= -2\sum_{\alpha,\beta=1}^d\int \frac{\partial m_2}{\partial v_{i,\alpha}}\frac{\partial^2 f}{\partial v_{i,\alpha}\partial v_{j,\beta}}\frac{\partial \log f}{\partial v_{j,\beta}}\\
    &=2\sum_{\alpha,\beta=1}^d\int \frac{\partial f}{\partial v_{j,\beta}}\frac{\partial}{\partial v_{i,\alpha}}\left(\frac{\partial m_2}{\partial v_{i,\alpha}}\frac{1}{f}\frac{\partial f}{\partial v_{j,\beta}}\right)\\
    &=2A^{i,j}+E^{i,j}-F^{i,j},
\end{align*}
so that
\[
F^{i,j} = A^{i,j}+\frac{E^{i,j}}{2},
\]
and, expanding,
\begin{align*}
    C^{i,j}& = 4\int m_2 f \ \nabla_{v_i} \log f \cdot \nabla^2_{v_i,v_j} \log m_2 \nabla_{v_j} \log f-2 F^{i,j}.
\end{align*}
So, gathering these two last identities,
\begin{align*}
   A^{i,j}+C^{i,j}+\frac{E^{i,j}}{2}+F^{i,j}& = 4\int m_2 f\  \nabla_{v_i} \log f \cdot \nabla^2_{v_i,v_j} \log m_2 \nabla_{v_j} \log f +A^{i,j}+\frac{E^{i,j}}{2}-F^{i,j}\\ 
   &= 4\int m_2 f\  \nabla_{v_i} \log f \cdot \nabla^2_{v_i,v_j} \log m_2 \nabla_{v_j} \log f,
\end{align*}
and finally
\begin{align}\label{eq: expression I2}
    I_2 =\sum_{i,j=1}^k -2\int m_2 f\left\Vert \nabla^2_{v_i,v_j} \log f\right\Vert^2 +4\int m_2 f\  \nabla_{v_i} \log f \cdot \nabla^2_{v_i,v_j} \log m_2 \nabla_{v_j} \log f.
\end{align}

\subsection{Dissipation of Fisher information: the drift term}

We focus in this subsection on the term $I_1$. Recalling \eqref{eq: dyn I(m1,m2)}, we get the decomposition
\begin{align*}
I_{1}=\sum_{i=1}^{k}\left(\sum_{j=1}^k \underbrace{-\int \nabla_{v_{j}} \cdot \left(b_{2}^{j,k} m_{2} \right) \frac{\norm{\nabla_{v_{i}} f}^{2}}{f}}_{=:G^{i,j}}+\underbrace{2 \int \frac{m_{2}}{f} \  \nabla_{v_{i}} \left(\mathcal{L}_{1} f\right) \cdot \nabla_{v_{i}} f}_{=:H^{i}}\ \underbrace{ -\int m_{2} \ \norm{\nabla_{v_{i}} f}^{2} \ 
\frac{\mathcal{L}_{1} f}{f^{2}}}_{= :J^{i}}\right).
\end{align*}
We will bound each term, making appear some crucial cancellations in the process. Let us first focus on $H^i$. Relying on \eqref{eq: cL1},
\begin{align*}
H^i=  \sum_{j=1}^{k}& \underbrace{2 \int \frac{m_{2}}{f} \ \nabla_{v_{i}} f \cdot \nabla_{v_{i}} \left(\nabla_{v_{j}} \cdot \left(\left(b_{2}^{j}-b_{1}^{j} \right) f\right)  \right)}_{=: H^{i,j}_{1}}\ \underbrace{ -2 \int \frac{m_2}{f}\nabla_{v_i} f \cdot \nabla_{v_i} \left(b_{2}^{j}\cdot \nabla_{v_j}  f\right)}_{=:H^{i,j}_2}
\\ &
+ \underbrace{2  \int m_{2} \frac{\nabla_{v_{i}} f}{f} \cdot \nabla_{v_{i}} \left(\left(b_{2}^{j}-b_{1}^{j}\right) \cdot \nabla_{v_{j}} \log m_{2}  f\right)}_{=:H^{i,j}_{3}}.
\end{align*}
We have, after an integration by parts:
\begin{align*}
H^{i,j}_{1}& =2\sum_{\alpha,\beta=1}^d \int \frac{m_2}{f}\frac{\partial f}{\partial v_{i,\alpha}}\frac{\partial}{\partial v_{i,\alpha}}\left(\frac{\partial }{\partial v_{j,\beta}}\left(\left(b^{j,\beta}_2-b^{j,\beta}_1\right)f\right)\right)
\\ 
&=-2 \sum_{\alpha,\beta=1}^d \int \frac{\partial}{\partial v_{j,\beta}}\left( \frac{m_2}{f}\frac{\partial f}{\partial v_{i,\alpha}}\right)\frac{\partial }{\partial v_{i,\alpha}}\left(\left(b^{j,\beta}_2-b^{j,\beta}_1\right)f\right)\\ 
&=\underbrace{-2\int m_2  f \nabla_{v_i} \log f \cdot \nabla_{v_i}\left(b_{2}^{j}-b_{1}^{j}\right) \nabla_{v_j}\log m_2}_{=:H^{i,j}_4}\\
&\quad \underbrace{-2\int m_2 f \left\Vert \nabla_{v_i} \log f\right\Vert^2 \left(b_{2}^{j}-b_{1}^{j}\right)\nabla_{v_j}\log m_2}_{=:H^{i,j}_5} \\
&\quad \underbrace{- 2\int m_2 f  \nabla^2_{v_i,v_j} \log f :\nabla_{v_{i}}\left(b^{j}_2-b^{j}_1\right)}_{=:H^{i,j}_6}\ \underbrace{- 2\int m_2 f \nabla_{v_i} \log f  \cdot\nabla^2_{v_i,v_j} \log f  \left(b^{j}_2-b^{j}_1\right) }_{=:H^{i,j}_7},
\end{align*}
where $A : B$ denotes the Frobenius inner product between two square matrices $A$ and $B$. Relying on the inequality $x \cdot y \leq \varepsilon \norm{x}^{2}+\frac{1}{4 \varepsilon} \norm{y}^{2}$,
\begin{align}\label{eq: bound H6}
   \sum_{i,j=1}^k  H^{i,j}_6 \leq C\sum_{i,j=1}^k \int m_1 \left\Vert \nabla_{v_i}\left(b_{2}^{j}-b_{1}^{j}\right)\right\Vert^2+\varepsilon \sum_{i,j=1}^k \int m_1 \left\Vert \nabla^2_{v_i,v_j} \log f\right\Vert^2,
\end{align}
The terms $H_{7}^{i,j}$,$H^{i,j}_4$ and $H^{i,j}_5$ will be cancelled with terms appearing in the next computation, while a cancellation will also occur for $H^{i,j}_2$ at the end of the proof.
Let us now focus on $H^{i,j}_{3}$. 
We have, by simple computations,
\begin{align*}
H^{i,j}_{3}&
=\underbrace{2\int m_{2} \frac{\nabla_{v_{i}} f}{f} \cdot \left(\nabla_{v_{i}} \left(b_{2}^{j}-b_{1}^{j}\right)\nabla_{v_{j}} \log m_{2}\right) f}_{=:H^{i,j}_{8}}
\\
&\quad +\underbrace{2  \int m_{2} \frac{\nabla_{v_{i}} f}{f} \cdot \nabla^2_{v_{i},v_{j}} \log m_{2}  \left(b_{2}^{j}-b_{1}^{j}\right) \ f}_{=: H^{i,j}_{9}}
\\ &
\quad +\underbrace{2 \int m_{2} \frac{\nabla_{v_{i}} f}{f}  \cdot \nabla_{v_{i}} f \left(b_{2}^{j}-b_{1}^{j}\right) \cdot \nabla_{v_{j}} \log m_{2}}_{=:H^{i,j}_{10}}.
\end{align*}
We remark immediately that
\[
H^{i,j}_4+H^{i,j}_8=0  \text{ and } H^{i,j}_5+ H^{i,j}_{10}=0.
\]
$H^{i,j}_{9}$ is a term that appears in our Lemma~\ref{lem: I(m1|m2)}.
We will now focus on the term $J^{i}$.
We have, recalling the expression of $\mathcal{L}_{1}f$,
\begin{align*}
J^i = \sum_{j=1}^{k}&\underbrace{ - \int m_{2} \frac{\norm{\nabla_{v_{i}} f}^{2}}{f^{2}} \nabla_{v_{j}} \cdot \left(\left(b_{2}^{j}-b_{1}^{j} \right) f\right)}_{=: J^{i,j}_{1}}+\underbrace{  \int m_{2} \frac{\norm{\nabla_{v_{i}} f}^{2}}{f^{2}} b_{2}^{j} \cdot \nabla_{v_{j}}  f}_{=: J^{i,j}_{2}}
\\ &
\underbrace{-\int m_{2} \frac{\norm{\nabla_{v_{i}} f}^{2}}{f^{2}} \left(b_{2}^{j}-b_{1}^{j}\right) \cdot \nabla_{v_{j}} \log m_{2} \ f}_{=: J^{i,j}_{3}}.
\end{align*}
Remark first that an integration by part leads to
\begin{align*}
    J^{i,j}_{1}& = -\int m_{2} \frac{\norm{\nabla_{v_{i}} f}^{2}}{f^{2}} \nabla_{v_{j}} \cdot \left(b_{2}^{j}-b_{1}^{j} \right)f-\int m_{2} \frac{\norm{\nabla_{v_{i}} f}^{2}}{f^{2}}  \left(b_{2}^{j}-b_{1}^{j} \right)\cdot \nabla_{v_{j}}  f\\ 
    & = \int  \nabla_{v_j}\left( m_{2}\frac{\norm{\nabla_{v_{i}} f}^{2}}{f^{}} \right)\cdot \left(b_{2}^{j}-b_{1}^{j} \right)f-\int m_{2} \frac{\norm{\nabla_{v_{i}} f}^{2}}{f^{2}}  \left(b_{2}^{j}-b_{1}^{j} \right)\cdot \nabla_{v_{j}}  f\\ 
    & = \underbrace{\int m_2 f \left\Vert \nabla_{v_i}\log f\right\Vert^2\nabla_{v_j} \log m_2 \cdot \left(b_{2}^{j}-b_{1}^{j}\right)}_{=:J^{i,j}_4}\\
    &\quad +\underbrace{2\int m_2 \frac{\nabla^2_{v_j,v_i} f \nabla_{v_i} f}{f}\cdot \left(b_{2}^{j}-b_{1}^{j}\right)-2 \int m_{2} \frac{\norm{\nabla_{v_{i}} f}^{2}}{f^{2}}  \left(b_{2}^{j}-b_{1}^{j} \right)\cdot \nabla_{v_{j}}  f}_{=:J^{i,j}_5},
\end{align*}
and note that
\[
J^{i,j}_3+J^{i,j}_4=0.
\]
Moreover, 
\begin{align*}
J^{i,j}_{5}=2 \int m_{2}\  \left(b_{2}^{j}-b_{1}^{j} \right)\cdot \left( \frac{\nabla^{2}_{v_{j},v_{i}} f \ \nabla_{v_{i}} f}{f}
-  \frac{\norm{\nabla_{v_{i}} f}^{2} }{f^{2}}\  \nabla_{v_{j}} f \right).
\end{align*}
However, by standard computation we know that
\begin{align*}
\nabla^{2}_{v_{j},v_{i}} \log f& =\frac{\nabla^{2}_{v_{j},v_{i}} f}{f}-\frac{\nabla_{v_{j}} f \otimes \nabla_{v_{i}} f}{f^{2}},
\end{align*}
so that
\begin{align*}
\nabla^{2}_{v_{j},v_{i}} \log f \ \nabla_{v_{i}} f=\frac{\nabla^{2}_{v_{i},v_{j}} f \ \nabla_{v_{i}} f}{f}
-\frac{\norm{\nabla_{v_{i}} f}^{2}}{f^{2}} \ \nabla_{v_{j}} f.
\end{align*}
Therefore:
\begin{align*}
 J^{i,j}_{5}& =2 \int m_{2} \ \left(( b_{2}^{j}-b_{1}^{j}\right) \cdot \nabla^{2}_{v_{j},v_{i}} \log f \ \nabla_{v_{i}} f,
\end{align*}
which implies $J_{5}^{i,j}+H_{7}^{i,j}=0$.

At last, it remains to remark that an integration by parts leads to
\begin{align*}
    G^{i,j} = 2\int \frac{m_2}{f}  b_{2}^{j}\cdot \nabla^2_{v_j,v_i} f \nabla_{v_i} f-J^{i,j}_2,
\end{align*}
while a straightforward computation gives
\begin{align*}
    H^{i,j}_2 = - 2 \int \frac{m_2}{f}  b_{2}^{j}\cdot \nabla^2_{v_j,v_i} f \nabla_{v_i} f - 2\int m_2 f \nabla_{v_i}\log f \cdot \nabla_{v_i} b_2^j \nabla_{v_j} \log f,
\end{align*}
and
\begin{align}\label{eq: bound G H J}
  \sum_{i,j=1}^k G^{i,j}+ H^{i,j}_2+J^{i,j}_2 = -2 \sum_{i,j=1}^{k} \int m_{2} f \ \nabla_{v_{i}} \log f \cdot \nabla_{v_{i}} b_{2}^{j} \ \nabla_{v_{j}} \log f.
\end{align}
Finally, gathering the estimates \eqref{eq: bound H6}, and \eqref{eq: bound G H J} we obtain, for all $\varepsilon >0$,
\begin{align}\label{eq: bound I1}
  \nonumber  I_1 & \leq 2  \sum_{i,j=1}^{k} \int m_{2} \frac{\nabla_{v_{i}} f}{f} \cdot \nabla^2_{v_{i},v_{j}} \log m_{2}  \left(b_{2}^{j}-b_{1}^{j}\right) \ f + C \sum_{i,j=1}^{k} \int m_{1} \norm{\nabla_{v_{i}} \left(b_{2}^{j}-b_{1}^{j} \right)}^{2}
    \\ &
   \quad -2 \sum_{i,j=1}^{k} \int m_{2} f \ \nabla_{v_{i}} \log f \cdot \nabla_{v_{i}} b_{2}^{j} \ \nabla_{v_{j}} \log f+ \varepsilon \sum_{i,j=1}^k \int m_1 \left\Vert \nabla^2_{v_i,v_j} \log f\right\Vert^2.
\end{align}

We conclude by showing that with our hypotheses all the terms appearing in this proof are well defined. For example,
\begin{equation*}
\left\vert A^{i,j}\right\vert \leq C \int \left(m_1\Vert \Delta \log m_2\Vert^2 + m_1 \Vert \nabla \log m_2\Vert^4+ m_2\frac{\norm{\nabla f}^4}{f^3}\right),
\end{equation*}
\begin{equation*}
    \left|G_{i,j}\right| \leq C\int\left( \left\Vert \nabla_{v_j}\cdot b_{2}^{j} \right\Vert_\infty m_2 \frac{\norm{\nabla f}^2}{f} + m_1 \norm{b_2}^4 +  m_1 \norm{\nabla \log m_2}^4 + m_2\frac{\norm{\nabla f}^4}{f^3}\right),
\end{equation*}
or
\begin{align*}
    \left|H_1^{i,j}\right| &\leq C\int \Bigg( \left\Vert \nabla^2 \left(b_{1}^{j}- b_{2}^{j}\right) \right\Vert_\infty 
    \left(m_1+m_2\frac{\norm{\nabla f}^2}{f} \right)+ \left\Vert \nabla \left(b_{1}^{j}-b_{2}^{j}\right) \right\Vert_\infty m_2\frac{\norm{\nabla f}^2}{f}\\
    &\qquad \qquad + 2 \left\Vert  b_{1}^{j}-  b_{2}^{j} \right\Vert_\infty \left(m_2 \frac{\norm{\nabla f}^2}{f} + m_2 \frac{\norm{\nabla^2 f}^2}{f} \right)\Bigg).
\end{align*}
The other terms can be treated in the same way.

However, in order to prove Theorem~\ref{th:main}, we will consider terms $b_1$ and $b_2$ such that, while $b_1-b_2$ and the derivatives of $b_2$ will still be bounded, the derivatives of $b_1$ will not be bounded (recall \eqref{eq:mukNt} and \eqref{eq:def bi1}). We will prove in Proposition~\ref{prop:derivatives b1} that, when replacing $b^j_1$ by $b^{j,k}_1$, the term appearing in the proof above involving the derivatives of $b^{j,k}_1$ are still well-defined.

\section{Proof of Theorem~\ref{th:main}}
\label{ProofThm}

We first start by choosing $\mu_{0}$ and $\mu_{0}^{k,N}$ such that $\mu_{0}^{\otimes N}$ and $\mu_{0}^{N,N}$ satisfy the assumption of Theorem~\ref{th:main} so that, in virtue of Proposition~\ref{prop:derivatives b1}, we can follow the proof of Lemma~\ref{lem: I(m1|m2)}, all the integration by parts we need being justified. By using Lemma~\ref{lem: I(m1|m2)}, we have
\begin{align*}
\frac{d}{dt} I_{t}^{k} &\leq C  I_{t}^{k}-(2-\varepsilon)\sum_{i,j=1}^k  \int \mk \left\Vert \nabla^2_{v_i,v_j} \log \fk \right\Vert^2  \\
&\quad +4\sum_{i,j=1}^k\int \mk \ \nabla_{v_i} \log \fk \cdot \nabla^2_{v_i,v_j} \log \mtk \nabla_{v_j} \log \fk
\\
&
\quad +2  \sum_{i,j=1}^{k} \int \mk \nabla_{v_{i}} \log \fk \cdot \nabla^2_{v_{i},v_{j}} \log \mtk \left(b_{2,t}^{j,k}-b_{1,t}^{j,k}\right) \\
&\quad + C \sum_{i,j=1}^{k} \int \mk \norm{\nabla_{v_{i}} \left(b_{2,t}^{j,k}-b_{1,t}^{j,k} \right)}^{2}
\\ &
\quad -2 \sum_{i,j=1}^{k} \int \mk \ \nabla_{v_{i}} \log \fk \cdot \nabla_{v_{i}} b_{2,t}^{j,k} \ \nabla_{v_{j}} \log \fk.
\end{align*}
Note that, because we used Lemma~\ref{lem: I(m1|m2)} with $m_{2}=\mtk$, there are additional cancellation. Indeed, because the particles are independent, we have, if $i \neq j$, $\nabla^{2}_{v_{i},v_{j}} \log \mtk=0$ and $\nabla_{v_{i}} b_{2}^{j,k}=0$. Therefore, we have, because $\nabla^{2}_{v_{i},v_{i}} \log \mtk$ is bounded by a constant (see Proposition~\ref{LemDev2}),
\begin{align*}
&\sum_{i,j=1}^k\int \mk \ \nabla_{v_i} \log \fk \cdot \nabla^2_{v_i,v_j} \log \mtk \nabla_{v_j} \log \fk\\ &\qquad\qquad\qquad\qquad=\sum_{i=1}^{k} \int \mk \ \nabla_{v_i} \log \fk \cdot \nabla^2_{v_i,v_i} \log \mtk \nabla_{v_j} \log \fk 
\\ &\qquad\qquad\qquad\qquad
\leq C I_{t}^{k},
\end{align*}
and 
\begin{align*}
\sum_{i,j=1}^{k} \int \mk \nabla_{v_{i}} \log \fk \cdot \nabla^2_{v_{i},v_{j}} \log \mtk \left(b_{2,t}^{j,k}-b_{1,t}^{j,k}\right) &
\leq C I_{t}^{k}+C \sum_{i=1}^{k} \int \mk \norm{b_{2,t}^{i,k}-b_{1,t}^{i,k}}^{2}.
\end{align*}
Because $\nabla \Gamma$ and $\nabla F$ are bounded, we also have
\begin{align*}
\left|\sum_{i,j=1}^{k} \int \mk \ \nabla_{v_{i}} \log \fk \cdot \nabla_{v_{i}} b_{2}^{j,k} \ \nabla_{v_{j}} \log \fk \right| \leq C I_{t}^{k}.
\end{align*}
Therefore, for $\varepsilon$ small enough, we have
\begin{align*}
\frac{d}{dt} I_{t}^{k} & \leq  
C   I_{t}^{k} +C  \underbrace{\sum_{i=1}^{k} \int \mk \ \norm{b_{2,t}^{i,k}-b_{1,t}^{i,k}}^{2}}_{\coloneq K}
+C\underbrace{   \sum_{i,j=1}^{k} \int \mk \ \norm{\nabla_{v_{i}} \left( b_{2,t}^{j,k}-b_{1,t}^{j,k} \right)}^{2}}_{\coloneq L}.
\end{align*}
By using the expressions of $b_{1,t}^{i,k}$ and $b_{2,t}^{i,k}$, the terms involving $F$ cancel out and we have:
\begin{align*}
K=\sum_{i=1}^{k} \int \mk \norm{\frac{1}{N} \sum_{j=1}^{k}\left( \Gamma(v_{i}-v_{j})-\Gamma\ast \bar\mu_{t}(v_{i}) \right)+\frac{N-k}{N} \langle \Gamma(v_{i}-\cdot), \mu_{t}^{k+1|k}-\bar\mu_{t} \rangle}^{2}.
\end{align*}
To bound the terms $K$ we follow here similar arguments as in \cite{lacker}: we get the decomposition (remark that if $k=N$ the second term is not present):
\begin{align*}
    K&\leq \underbrace{\frac{2}{N^2}\sum_{i=1}^k \int \mk \norm{\sum_{j=1}^{k}\left( \Gamma(v_{i}-v_{j})-\Gamma\ast \bar\mu_{t}(v_{i}) \right)}^2}_{:=K_1} + \underbrace{2 \sum_{i=1}^k\int \mk \norm{ \langle \Gamma(v_{i}-\cdot), \mu_{t}^{k+1|k}-\bar\mu_{t} \rangle}^2}_{:=K_2},
\end{align*}
and Pinsker inequality leads to
\[
\norm{ \langle \Gamma(v_{i}-\cdot), \mu_{t}^{k+1|k}(\cdot|v_1,\ldots v_k)-\bar\mu_{t}(\cdot) \rangle}\leq \norm{\Gamma}_\infty \sqrt{H\left(\mu_{t}^{k+1|k}(\cdot|v_1,\ldots v_k)|\bar\mu_{t}(\cdot)\right)},
\]
so, relying on the towering property of the relative entropy, we obtain the bound
\[
K_2 \leq Ck  \left(H_{t}^{k+1}-H_{t}^{k} \right){\bf 1}_{k<N}.
\]
Concerning $K_1$, we have, relying on the exchangeability of the particles,
\begin{align*}
    K_1 \leq 4\frac{k^2}{N^2} \norm{\Gamma}_\infty + 4\frac{k^3}{N^2}\int \mu^{3, N}_t \left(\Gamma(v_1-v_2)-\Gamma*\bar\mu_t(v_1)\right)\left(\Gamma(v_1-v_3)-\Gamma*\bar\mu_t(v_1)\right),
\end{align*}
so that, remarking that
\[
\int \bar \mu^{\otimes 3}_t \left(\Gamma(v_1-v_2)-\Gamma*\bar\mu_t(v_1)\right)\left(\Gamma(v_1-v_3)-\Gamma*\bar\mu_t(v_1)\right)=0,
\]
relying again on Pinsker inequality and the fact that $\Gamma$ is bounded we get
\[
K_1 \leq C\frac{k^2}{N^2} + C\frac{k^3}{N^2}\sqrt{H\left(\mu^{3,N}_t|\bar \mu^{\otimes 3}_t\right)}.
\]
Applying the relative entropy bound given by \cite{lacker} we get $H\left(\mu^{3,N}_t|\bar \mu^{\otimes 3}_t\right)\leq \frac{C}{N^2}$, so that, gathering the previous estimates,
\[
K \leq C\frac{k^2}{N^2} + C k (H^{k+1}_t-H^k_t){\bf 1}_{k<N}.
\]
We focus now on the term $L$. Notice that, because $b_{2,t}^{j,k}(v_{1},...,v_{k})=\Gamma \ast \mu_{t}(v_{j})+F(v_{j})$, we have, if $i \neq j$, $\nabla_{v_{i}} b_{2,t}^{j,k}=0$. Therefore, we have:
\begin{align*}
L& =\underbrace{\sum_{j=1}^{k} \int \mk \ \norm{\nabla_{v_{j}} \left( b_{2,t}^{j,k}-b_{1,t}^{j,k} \right)}^{2}}_{=L_1}
+\underbrace{\sum_{i \neq j} \int \mk \ \norm{\nabla_{v_{i}} \left( b_{2,t}^{j,k}-b_{1,t}^{j,k} \right)}^{2}}_{= L_{2}}.
\end{align*}
Notice moreover that, if $i \neq j$, we have
\[
\nabla_{v_{i}} \left( b^{j,k}_2-b_{1,t}^{j,k}\right)=-\frac{1}{N} \nabla \Gamma(v_{j}-v_{i})-\frac{N-k}{N} \langle \Gamma(v_{j}-\cdot), \nabla_{v_{i}} \mu_{t}^{k+1|k} \rangle.
\]
Therefore, relying on the fact that $\Gamma$ is bounded (again, if $k=N$, the second term is not present of the right-hand side),
\begin{align*}
L_{2} & \leq 2  \sum_{i \neq j} \int \mk \norm{\frac{1}{N} \nabla \Gamma(v_{j}-v_{i})}^{2}+2 \ \sum_{i \neq j} \int \mk \norm{\frac{N-k}{N} \langle \Gamma(v_{j}-\cdot), \nabla_{v_{i}} \mu_{t}^{k+1|k} \rangle}^{2}
\\ & 
\leq C \frac{k^{2}}{N^{2}}+C \sum_{i,j=1}^{k} \int \mk \norm{\langle \Gamma(v_{j}-\cdot), \nabla_{v_{i}} \mu_{t}^{k+1|k}}^{2}
\\ &
\leq C \frac{k^{2}}{N^{2}}+ C  \sum_{i,j=1}^{k} \int \mk \norm{\int \mu_{t}^{k+1|k}(dv^{*}) \ \Gamma(v_{i}-v^{*}) \cdot \nabla_{v_{i}} \log \mu_{t}^{k+1|k} }^{2}
\\ &
\leq C \frac{k^{2}}{N^{2}}+ C  \sum_{i,j=1}^{k}\int \mk \int \mu_{t}^{k+1|k}(dv^{*}) \ \norm{ \Gamma(v_{i}-v^{*}) \cdot \nabla_{v_{i}} \log \mu_{t}^{k+1|k}}^{2}
\\ &
\leq C \frac{k^{2}}{N^{2}}+ C \sum_{i,j=1}^{k} \int \mu_{t}^{k+1,N} \ \norm{ \nabla_{v_{i}} \log \mu_{t}^{k+1|k}}^{2}.
\end{align*}
Notice moreover that, by exchangeability, we have
\begin{align*}
I_{t}^{k+1}& =\sum_{i=1}^{k+1} \int \mu_{t}^{k+1,N} \norm{\nabla_{v_{i}} \log \frac{\mu_{t}^{k+1,N}}{\mu^{\otimes k+1}}}^{2}
\\ &
=\frac{1}{k+1} I_{t}^{k+1}+\sum_{i=1}^{k} \int \mu_{t}^{k+1,N} \norm{\nabla_{v_{i}} \log \fk+ \nabla_{v_{i}} \log \frac{\mu_{t}^{k+1|k}}{\bar\mu_{t}(v_{k+1})}}^{2},
\end{align*}
and thus, since $\nabla_{v_{i}} \left(\log \mu(v_{k+1}) \right)=0$,
\begin{align*}
I_{t}^{k+1}& =\frac{1}{k+1} I_{t}^{k+1}+\sum_{i=1}^{k} \int \mu_{t}^{k+1,N} \norm{\nabla_{v_{i}} \log \fk+ \nabla_{v_{i}} \log \mu_{t}^{k+1|k}}^{2}
\\ &
=\frac{1}{k+1} I_{t}^{k+1}+\sum_{i=1}^{k} \int \mu_{t}^{k+1,N} \norm{\nabla_{v_{i}} \log \fk}^{2}+\sum_{i=1}^{k} \int \mu_{t}^{k+1,N}  \norm{\nabla_{v_{i}} \log \mu_{t}^{k+1|k}}^{2}
\\ &\quad  +
2 \ \sum_{i=1}^{k} \int \mu_{t}^{k+1,N} \nabla_{v_{i}} \log \mu_{t}^{k+1|k} \cdot \nabla_{v_{i}} \log \fk
\\ &
=\frac{1}{k+1} I_{t}^{k+1}+I_{t}^{k}+\sum_{i=1}^{k} \int \mu_{t}^{k+1,N}  \norm{\nabla_{v_{i}} \log \mu_{t}^{k+1|k}}^{2}\\ 
&\quad +2 \ \sum_{i=1}^{k} \int \mu^{k,N} \left(\int \mu_{t}^{k+1|k} \nabla_{v_{i}} \log \mu_{t}^{k+1|k} \right) \cdot \nabla_{v_{i}} \log \fk.
\end{align*}
Noticing that
\begin{align*}
\int \mu_{t}^{k+1|k} \nabla_{v_{i}} \log \mu_{t}^{k+1|k}& =\int \nabla_{v_{i}} \mu_{t}^{k+1|k}
\\ &
=\nabla_{v_{i}} \int \mu_{t}^{k+1|k}
\\ &
=0,
\end{align*}
we deduce
\begin{align*}
I_{t}^{k+1}-I_{t}^{k} \geq \sum_{i=1}^{k} \int \mu_{t}^{k+1,N}  \norm{\nabla_{v_{i}} \log \mu_{t}^{k+1|k}}^{2}.
\end{align*}
Using this estimate yields the bound 
\begin{align*}
L_{2} \leq C  \frac{k^{2}}{N^{2}}+Ck \left(I_{t}^{k+1}-I_{t}^{k} \right){\bf 1}_{k<N}.
\end{align*}
It remains to bound $L_1$:
\begin{align*}
L_1 & =C \sum_{j=1}^{k} \int \mk \norm{\nabla \Gamma \ast \mu_{t}(v_{j})-\frac{1}{N} \sum_{i=1}^{k} \nabla \Gamma(v_{j}-v_{i})-\frac{N-k}{N} \nabla_{v_{j}} \langle \Gamma(v_{j}-\cdot), \mu_{t}^{k+1|k} \rangle}^{2}
\\ &
=C \sum_{j=1}^{k} \int \mk \left\|\frac{1}{N} \sum_{i=1}^{k} \nabla \Gamma(v_{j}-v_{i})-\nabla \Gamma\ast \mu_{t}(v_{j})+\frac{N-k}{N} \langle \nabla \Gamma(v_{j}-\cdot), \mu_{t}-\mu_{t}^{k+1|k} \rangle\right.\\
&\qquad\qquad\qquad\qquad\qquad \left.-\frac{N-k}{N} \langle \Gamma(v_{j}-\cdot, \nabla_{v_{j}} \mu_{t}^{k+1|k} \rangle\right\|^{2}
\\ &
\leq C \sum_{j=1}^{k} \int \mk \norm{\frac{1}{N} \sum_{i=1}^{k} \nabla \Gamma(v_{j}-v_{i})-\nabla \Gamma\ast \mu_{t}(v_{j})+\frac{N-k}{N} \langle \nabla \Gamma(v_{j}-\cdot), \mu_{t}-\mu_{t}^{k+1|k} \rangle}^{2}
\\ &
 \quad +C\sum_{j=1}^{k} \int \mk \norm{\langle \Gamma(v_{j}-\cdot), \nabla_{v_{j}} \log \mu_{t}^{k+1|k} \rangle}^{2}.
\end{align*}
By similar arguments as before, we get the bound
\begin{align*}
L_2\leq C \frac{k^{2}}{N^{2}}+Ck\left(H_{t}^{k+1}-H_{t}^{k} \right){\bf 1}_{k<N}+C\left(I_{t}^{k+1}-I_{t}^{k} \right){\bf 1}_{k<N}.
\end{align*}
Finally, gathering the previous estimates, we obtain
\begin{align*}
\frac{d}{dt}I_{t}^{k} & \leq C \frac{k^{2}}{N^{2}}+Ck\left(H_{t}^{k+1}-H_{t}^{k} \right){\bf 1}_{k<N}+Ck\left(I_{t}^{k+1}-I_{t}^{k} \right){\bf 1}_{k<N}+C \ I_{t}^{k}.
\end{align*}
To obtain a bound on $I_{t}^{k}$ we also consider the hierarchy provided by the relative entropy, as was done in \cite{lacker, lacker_uniforme}: standard computations (see for example \cite[Lemma 3.1.]{lacker_uniforme}) give
\[
\frac{d}{dt}H_{t}^{k}=- I_{t}^{k}+\sum_{j=1}^k \int \mu^{k,N}_t \norm{b^{j,k}_2 - b^{j,k}_1}^2,
\]
so, with similar estimates as before (remark in particular that with our hypotheses $b^{j,k}_2 - b^{j,k}_1$ is bounded) we obtain
\begin{align*}
\frac{d}{dt}H_{t}^{k}\leq - I_{t}^{k}+C \ \frac{k^{2}}{N^{2}}+Ck  \left(H_{t}^{k+1}-H_{t}^{k} \right){\bf 1}_{k<N}.
\end{align*}
Let $\alpha>0$ and $z_{t}^{k} \coloneqq I_{t}^{k}+\alpha H_{t}^{k}$. We obtain, for $\alpha$ large enough:
\begin{align*}
\frac{d}{dt} z_{t}^{k} & \leq C \frac{k^{2}}{N^{2}}+ \left(C-\alpha \right) I_{t}^{k}+C k \left(H_{t}^{k+1}-H_{t}^{k} \right){\bf 1}_{k<N}+C k \left(I_{t}^{k+1}-I_{t}^{k} \right){\bf 1}_{k<N}.
\\ & 
\leq C\frac{k^{2}}{N^{2}}+Ck\left(z_{t}^{k+1}-z_{t}^{k} \right){\bf 1}_{k<N}.
\end{align*}
These are exactly the same equations as obtained \cite{lacker}, in the context of the relative entropy. Relying on the estimates obtained in this paper, we deduce that
\[ \forall t \in \mathbb{R}_{+},  \ z_{t}^{k} \leq Ce^{Ct} \frac{k^{2}}{N^{2}}.\]
This means in particular that
\[ I_{t}^{k} \leq Ce^{Ct} \frac{k^{2}}{N^{2}}.\]

\section{A Gaussian example: optimality of our rates}
\label{Gauss}

The main goal of this Section is to write down an explicit example where the rate we obtained previously in our main Theorem is shown to be sharp. It will be a Gaussian case, i.e. linear confinement and interaction potential. Note that this case lies beyond our general setting, i.e. $\Gamma$ bounded. Consider constants $a$ and $b$ and the system of $n$ interacting particles in $\mathbb{R}$ of the form
\begin{align*}
dX_{t}^{i}=-\left(a X_{t}^{i}+\frac{b}{n-1} \sum_{j \neq i} X_{t}^{j} \right)dt+dW_{t}^{i}.
\end{align*}
We define
\[
v(t) \coloneqq \frac{1}{2a} \left(1-e^{-2at} \right), \qquad c_{n}(t)=\frac{1}{n} \left(1-\frac{a}{a+\frac{bn}{n-1}} \frac{1-e^{-2t \left(a+\frac{bn}{n-1} \right) }}{1-e^{-2at}} \right).
\]
The $k$ particles marginal $\mu_{t}^{k,n}$ is a centered gaussian with covariance matrix $\Sigma_{t}^{n,k}$ satisfying $\Sigma_{t}^{n,k}=v(t)\left(I_{k}-c_{n}(t)J_{k} \right)$, where $I_k$ is the identity matrix and $J^k$ is the matrix of all ones. $\mu_{t}^{k,n}$ converges to $\mu_{t}^{ \otimes k}$, where $\mu_{t}$ is a centered gaussian with variance $v(t)$. Recall that the Fisher information between two gaussian random variables (whose covariance matrices commute) is given by:
\[I \left(\mathcal{N}\left(\mu_{1}, \Sigma_{1} \right)|\mathcal{N}\left(\mu_{2}, \Sigma_{2} \right) \right)=\norm{\Sigma_{2}^{-1} \left(\mu_{1}-\mu_{2} \right)}^{2}+\text{Tr}\left( \Sigma_{2}^{-2} \left(\Sigma_{2}-\Sigma_{1} \right)^{2} \Sigma_{1}^{-1} \right).\]
Relying on this formula, we obtain
\begin{align*} I_{t}^{k} & =\frac{1}{v(t)^{2}} \text{Tr} \left( v(t)^{2} c_{n}(t)^{2} J_{k}^{2} \left( v(t) \left(I_{k}-c_{n}(t) J_{k} \right) \right)^{-1} \right)
=\frac{1}{v(t)} c_{n}(t)^{2} k \ \text{Tr}\left(J_{k} \left(I_{k}-c_{n}(t) J_{k} \right)^{-1} \right).
\end{align*}
Note moreover that, since $J_{k}^{2}=k J_{k}$, $J_{k}$ is similar to the matrix 
\begin{equation*}
\begin{pmatrix}
k & 0 & \cdots & 0 \\
0 & 0 & \cdots & 0 \\
\vdots  & \vdots  & \ddots & \vdots  \\
0 & 0 & \cdots & 0 
\end{pmatrix},
\end{equation*}
and thus we have
\begin{align*} I_{t}^{k} = \frac{c_{n}(t)^{2} k^{2}}{v(t)(1-c_{n}(t) k)}.
\end{align*}
Noticing that $\frac{\alpha(t)}{n}\leq c_n(t) \leq \frac{\beta(t)}{n}$, with
\[
\alpha(t) = 1-\frac{a}{a+b} \frac{1-e^{-2(a+b)t}}{1-e^{-2at}},\quad \text{and} \quad \beta(t) = 1-\frac{a}{a+2b} \frac{1-e^{-2(a+2b)t}}{1-e^{-2at}},
\]
we deduce the bound
\[
\frac{\alpha^2(t)}{v(t)}\frac{k^2}{n^2}\leq I^k_t\leq \frac{\beta^2(t)}{v(t)(1-\beta(t))}\frac{k^2}{n^2}.
\]
Remark that in this estimate, $a$ and $b$ can be positive or negative in order to get a non uniform in time estimate. However if $a>0$ and $a+b>0$ then the propagation of chaos with sharp rate is uniform in time. We will consider the uniform in time case in a future work.

\section{Regularity estimates}
\label{Reg}

The aim of this Section is to justify that all the computations made in the proof of Lemma~\ref{lem: I(m1|m2)} are valid when replacing $m_{1,t}$ with \begin{equation*}
    m^k_{1,t} :=\mu^{k,N}_t
\end{equation*}
and $m_{2,t}$ with 
\begin{equation*}
    m^k_{2,t}:=\bar \mu^{\otimes k}_t.
\end{equation*}
On one hand, we will prove that, under suitable assumptions on the initial conditions, the hypotheses of Lemma~\ref{lem: I(m1|m2)}, at the exceptions of the regularity hypothesis on the derivatives of $b_1$, are satisfied. 
More precisely, denoting $f_{k,t} = m^k_{1,t}/m^k_{2œ,t}$ and
\begin{align*}
    I^k_t = \int m^k_{2,t} \frac{\Vert \nabla f_{k,t}\Vert^2}{f_{k,t}}, \quad J^k_t = \int m^k_{2,t} \frac{\Vert \nabla^2 f_{k,t}\Vert^2}{f_{k,t}},
    \quad K^k_t = \int m^k_{2,t} \frac{\Vert \nabla^3 f_{k,t}\Vert^2}{f_{k,t}},\\
    L^k_t = \int m^k_{2,t} \frac{\Vert \nabla f_{k,t}\Vert^4}{f_{k,t}^3},\quad
    M^k_t= \int m_{1,t}^k \norm{\nabla^2 \log m^k_{2,t}}^2 ,\quad N^k_t = \int m_{1,t}^k \norm{\nabla \log m^k_{2,t}}^4,\\
    O^k_t= \int m_{1,t}^k \norm{b_{2,t}}^4.
\end{align*}
We will prove the following Proposition.
\begin{proposition}\label{prop:Ik...Nk}
Suppose that there exist positive constants $\kappa_{i,0}$ and $\kappa_{i,0}'$ such that, for $i=1,2$,
\begin{equation}\label{hyp:gaussian tail I}
  \kappa_{i,0}^{-1} -(\kappa_{i,0}')^{-1} \norm{v}^2 \leq \log m^N_{i,0}(v) \leq \kappa_{i,0} -\kappa_{i,0}' \norm{v}^2,  
\end{equation}
and that there exist positive constants $\kappa_{i,j}$ such that, for $i=1,2$ and $j=1,2,3$,
\begin{equation}\label{hyp:gassian tail II bis}
    \norm{\nabla^j \log m^N_{i,0}}^2 \leq \kappa_{i,j}\left(1+ \left|\log m^N_{i,0}\right|^j\right).
\end{equation}
Then
\[
\max_{k=1,\ldots,N} \sup_{t\in [0,T]} (I^k_t + J^k_t + K^k_t + L^k_t + M^k_t + N^k_t + O^k_t) < \infty.
\]
\end{proposition}

\begin{rem}
The hypotheses of Proposition~\ref{prop:Ik...Nk} imply in particular that there exist constants $c_{i,j}$ and $c'_{i,j}$ such that, for $i=1,2$, $j=1,2,3$ and all $v\in \mathbb{R}^d$,
\begin{equation}\label{hyp:gassian tail II}
    \norm{\nabla^j  m^N_{i,0}}^2 \leq  c_{i,j} \left(1+ \left|\log m^N_{i,0}\right|^j\right) (m^N_{i,0})^2,
\end{equation}
and
\begin{equation}\label{hyp:gassian tail III}
    \norm{\nabla^j m^N_{i,0}}(v) \leq c'_{i,j}\left(1+\norm{v}^j\right) m_{i_0}^N(v).
\end{equation}
\end{rem}

On the other hand we will prove that the terms involving the derivatives of $b_1$ in the proof of Lemma~\ref{lem: I(m1|m2)} (this concerns the terms $H^{i,j}_1$, $H^{i,j}_3$, $H^{i,j}_4$, $H^{i,j}_6$, $H^{i,j}_8$ and $J^{i,j}_1$) are still well defined when $b_1$ is replaced by $b^{i,k}_1$ (recall \eqref{eq:def bi1}). More precisely, dropping the dependency in time in our notations, we will prove the following Proposition.

\begin{proposition}\label{prop:derivatives b1}
We have
\begin{equation*}
    \left| \int m^k_2 \nabla_{v_i} f_k  \cdot \nabla_{v_i}\left(\nabla_{v_j}\cdot\left(b^{j,k}_1 \right)\right)\right|\leq C\left(I^{k+1}+J^{k+1}+L^{k+1}+M^{k+1}+N^{k+1}\right),
\end{equation*}
\begin{equation*}
    \left| \int m^k_2 \frac{\nabla_{v_i} f_k}{f_k}  \cdot \nabla_{v_i} b^{j,k}_1 \nabla_{v_j} f_k\right|\leq C\left(I^{k+1}+L^{k+1}+N^{k+1}\right),
\end{equation*}
\begin{equation*}
    \left|\int m^k_2 \nabla_{v_i} f_k\cdot \nabla_{v_i} b^{j,k}_1 \nabla_{v_j} \log m^{k}_2 \right|\leq C\left(I^{k+1}+L^{k+1}+M^{k+1}+N^{k+1}\right),
\end{equation*}
\begin{equation*}
    \left|\int m^k_2 f_k\nabla^2_{v_i,v_j}\log f_k : \nabla_{v_i} b^{j,k}_1  \right|\leq C\left(I^{k+1}+J^{k+1}+L^{k+1}+N^{k+1} \right),
\end{equation*}
\begin{equation*}
    \left|\int m^k_2 \frac{\norm{\nabla_{v_i} f_k}^2}{f_k} \nabla_{v_j}\cdot  b^{j,k}_1  \right|\leq C\left( I^{k+1}+L^{k+1}+N^{k+1}\right).
\end{equation*}
\end{proposition}

\subsection{Proof of Proposition~\ref{prop:Ik...Nk}}

We begin with two Lemma showing hierarchies between the terms appearing in Proposition~\ref{prop:Ik...Nk}. The proof of Lemma~\ref{lem:link Lk Jk Nk} relies on similar arguments as in \cite[Proposition B.1.]{Tabary2025propagation}.

\begin{lemma}\label{lem:hierarchy I J K}
We have, for all $k=1,2,\ldots, N-1$,
\[
I^k \leq I^{k+1}, \quad J^k \leq J^{k+1}, \quad K^k \leq K^{k+1}.
\]
\end{lemma}
\begin{proof}
Let us prove it for $J^k$. We denote by $\nabla_k$ and $\nabla^2_k$ the restriction of the operators $\nabla$ and $\nabla^2$ to the variables $v_1,\ldots, v_k$. Remark first that
\begin{align*}
    J^{k+1} &\geq  \int m^{k+1}_2 \frac{\Vert \nabla^2_{k} (f_{k}f_{k+1|k})\Vert^2}{f_{k+1}}.
\end{align*}
Moreover, with the identity
\[
\nabla^2_{k} (f_{k}f_{k+1|k}) = f_{k+1|k} \nabla^2_{k} f_{k}+\nabla_k f_k \otimes \nabla_k f_{k+1|k}+\nabla_k f_k \otimes \nabla_k f_{k+1|k} +f_{k} \nabla^2_k f_{k+1|k},
\]
we obtain
\begin{align}
\int m^{k+1}_2 &\frac{\Vert \nabla^2_k (f_{k}f_{k+1|k})\Vert^2}{f_{k+1}}\nonumber\\
&\geq \int m_2^{k+1} f_{k+1|k}^2 \frac{\Vert \nabla^2_k f_{k}\Vert^2}{f_{k+1}}+2\int m_2^{k+1}\frac{f_{k+1|k} \nabla^2_kf_k  \nabla_k f_k \cdot \nabla_k f_{k+1|k}}{f_{k+1}}\nonumber\\
    &\quad + 2\int m_2^{k+1}\frac{f_{k+1|k} \nabla^2_kf_k  \nabla_k f_k \cdot \nabla_k f_{k+1|k} }{f_{k+1}}\nonumber\\
    &\quad + 2\int m_2^{k+1}\frac{f_{k} f_{k+1|k} \nabla^2_kf_k :  \nabla^2_k f_{k+1|k} }{f_{k+1}}\label{eq:decomp Jk}.
\end{align}
Remark now that
\begin{align*}
    \int m_2^{k+1} f_{k+1|k}^2 \frac{\Vert \nabla^2_k f_{k}\Vert^2}{f_{k+1}}& = \int m_2^k \frac{\Vert \nabla^2_k f_{k}\Vert^2}{f_{k}}\int m_2^{k+1|k} (v_{k+1}) f_{k+1|k}(v_{k+1}|v_1\ldots,v_k)\\
    & = \int m_2^k \frac{\Vert \nabla^2_k f_{k}\Vert^2}{f_{k}}\int m_1^{k+1|k}(v_{k+1}|v_1\ldots,v_k)\\
    & = J^k,
\end{align*}
while
\begin{align*}
  \int m_2^{k+1}&\frac{f_{k+1|k} \nabla^2_kf_k  \nabla_k f_k \cdot \nabla_k f_{k+1|k}}{f_{k+1}}\\
&=\int m_2^{k}\frac{\nabla^2_kf_k   \nabla_k f_k  }{f_{k}} \cdot \int  \nabla_k \big(m_2^{k+1|k} f_{k+1|k}\big)(v_{k+1}|v_1,\ldots, v_k)\\
&=\int m_2^{k}\frac{\nabla^2_kf_k  \nabla_k f_k  }{f_{k}} \cdot \int  \nabla_k  m_1^{k+1|k}(v_{k+1}|v_1,\ldots, v_k)\\
&=0.
\end{align*}
With similar computations one shows that the last two terms of the right-hand side of \eqref{eq:decomp Jk} also disappear. This proves that $J^k\leq J^{k+1}$. The proofs for $I^k$ and $K^k$ are similar.
\end{proof}

\begin{lemma}\label{lem:link Lk Jk Nk}
There exists a positive constant $C$ such that  
\[
L^k \leq C\left(J^k +N^k \right).
\]
\end{lemma}
\begin{proof}
Define, for $\beta\in (0,1)$,
\[
J^k_\gamma =  \int m_2^k \frac{\Vert \nabla^2 f^\gamma_k\Vert^2}{\gamma^2 f_k^{2\gamma-1}}.
\]
Since
\[
 \nabla^2 f^\gamma_k = \gamma f^{\gamma-1}\nabla^2 f_k+\gamma(\gamma-1)f_k^{\gamma-2}\nabla f_k\otimes \nabla f_k, 
\]
we obtain
\[
J^k_\gamma = J^k +(\gamma-1)^2 L^k + 2(\gamma-1) \int m_2^k \frac{\nabla^2 f_k  \nabla f_k\cdot \nabla  f_k}{f_k^2}.
\]
An integration by parts leads to
\begin{align*}
    \int m_2^k &\frac{1}{f_k^2}\frac{\partial^2 f_k}{\partial v_{i,\alpha}\partial v_{j,\beta}}\frac{\partial f_k}{\partial v_{i,\alpha}}\frac{\partial f_k}{\partial v_{j,\beta}}\\
    & = -\int m^2_k \frac{1}{f_k^2} \frac{\partial \log m_2^k}{\partial v_{i,\alpha}}\frac{\partial f_k}{\partial v_{i,\alpha}}\left(\frac{\partial f_k}{\partial v_{j,\beta}}\right)^2+2\int m_2^k \frac{1}{f_k^3}\left(\frac{\partial f_k}{\partial v_{i,\alpha}}\right)^2\left(\frac{\partial f_k}{\partial v_{j,\beta}}\right)^2\\
    &\quad   - \int m_2^k \frac{1}{f_k^2}\frac{\partial^2 f_k}{\partial v_{i,\alpha}^2}\left(\frac{\partial f_k}{\partial v_{j,\beta}}\right)^2
    - \int m_2^k \frac{1}{f_k^2}\frac{\partial^2 f_k}{\partial v_{i,\alpha}\partial v_{j,\beta}}\frac{\partial f_k}{\partial v_{i,\alpha}}\frac{\partial f_k}{\partial v_{j,\beta}}
\end{align*}
so that
\begin{align*}
    \int m_2^k \frac{\nabla^2 f_k  \nabla f_k\cdot \nabla  f_k}{f_k^2} = L^k - \frac12 \int m_2^k \frac{\Vert\nabla f_k\Vert^2  \nabla \log m_2^k\cdot \nabla  f_k}{f_k^2}-\frac12 \int m_2^k \frac{\Vert\nabla f_k\Vert^2  {\rm Tr}(\nabla^2 f_k)}{f_k^2}.
\end{align*}
We deduce that
\begin{align*}
    \big(2(1-\gamma)-(1-\gamma)^2\big)L^k &=J^k-J^k_\gamma +(1-\gamma) \int m_2^k \frac{\Vert\nabla f_k\Vert^2  \nabla \log m_2^k\cdot \nabla  f_k}{f_k^2}\\
    &\quad +(1-\gamma)\int m_2^k \frac{\Vert\nabla f_k\Vert^2  {\rm Tr}(\nabla^2 f_k)}{f_k^2}.
\end{align*}
Relying moreover on the bounds
\begin{align*}
   \left| \int m_2^k \frac{\Vert\nabla f_k\Vert^2  \nabla \log m_2^k\cdot \nabla  f_k}{f_k^2}\right| &\leq \frac{L^k}{2}+\frac12\int m_2^k \frac{\Vert \nabla \log m_2^k\Vert^2 \Vert \nabla f_k\Vert^2 }{f_k}\\
   &\leq \frac{ 3 L^k}{4}+\frac{N^k}{4} ,
\end{align*}
and
\begin{align*}
    \left|\int m_2^k \frac{\Vert\nabla f_k\Vert^2  {\rm Tr}(\nabla^2 f_k)}{f_k^2}\right| \leq \frac{L^k}{2}+ \frac{k^2 d^2 J^k}{2} ,
\end{align*}
we obtain
\begin{align*}
    \left(\frac{3(1-\gamma)}{4}-(1-\gamma)^2\right)L^k\leq \left(1+\frac{(1-\gamma)k^2 d^2}{2}\right)J^k+ \frac{1-\gamma}{4}N^k,
\end{align*}
which concludes the proof since the constant in front of $L^k$ is positive when $\gamma$ is sufficiently close to $1$.
\end{proof}

The following Lemma, based on the parabolic maximum principle, show that the regularity hypotheses made on the initial condition can be propagated for $t\in [0,T]$. The proof of this result is similar to the one given in \cite{feng2023quantitative} in the context of the $2d$ viscous vortex model, and to the one given in \cite{carrillo2024relative} for the Landau equation for Maxwellian molecules. 

\begin{lemma} \label{lemme_grad}
Under the hypotheses of Proposition~\ref{prop:Ik...Nk}, for any $T>$, there exists positive constants $\rho_{i,j}$ such that for $i=1,2$, $j=1,2,3$, $v\in \mathbb{R}^d$ and $t\in [0,T]$,
\[
\norm{\nabla^j m^N_{i,t}} \leq \rho_{i,j} \left(1+\norm{v}^j\right) m^N_{i,t}(v).
\]
\end{lemma}

\begin{proof}
Notice that, taking $k=N$, the coefficients $b^{i,N}_{1,t}$ and $b^{i}_{2,t}$ involved in the dynamics of $m^N_{1,t}=\mu^{N,N}_t$ and $m^N_{2,t} = \bar \mu^{\otimes N}_t$ are smooth with bounded derivatives (it is not the case for $b^{i,k}_{1,t}$ for $k<N$). For simplicity, since the proof is similar for $m^N_{1,t}$ and $m^{N}_{2,t}$, we will drop in our notations the dependency in $i$ and $N$ and consider a solution $m=m_t$ to
\[\partial_{t} m_{t}=\Delta m_{t}+\nabla \cdot \left( b_{t} m_{t} \right)=\mathcal{L} m_{t}, \]
with $b_t$ smooth with bounded derivatives and $m_0$ satisfying the hypotheses of Proposition~\ref{prop:Ik...Nk}.

\medskip

{\it First step: bound on $\norm{\nabla m_t}$.}
A straightforward computation leads to
\begin{align*}
\left(\partial_{t}-\mathcal{L}\right) \frac{\norm{\nabla m}^{2}}{m} & = 2 \frac{\nabla m \cdot \nabla \partial_{t} m}{m}-\partial_{t} m \frac{\norm{\nabla m}^{2}}{m^{2}}-\Delta \left(\frac{\norm{\nabla m}^{2}}{m} \right)-\nabla \cdot \left(b_{t} \frac{\norm{\nabla m}^{2}}{m} \right)
\\ &
=\underbrace{2 \nabla m \cdot \frac{\nabla \left[\Delta m \right]}{m}-\Delta\left(\frac{\norm{\nabla m}^{2}}{m} \right)-\Delta m \frac{\norm{\nabla m}^{2}}{m^{2}}}_{=A} \\ &
\quad 
+\underbrace{2\nabla m \cdot \frac{\nabla \left[ \nabla \cdot \left(b_{t} m \right) \right]}{m}
-\nabla \cdot \left(b_{t} m \right) \frac{\norm{\nabla m}^{2}}{m^{2}}-\nabla \cdot \left(b_{t} \frac{\norm{\nabla m}^{2}}{m} \right)}_{=B}.
\end{align*}
Notice that 
\begin{align*}
A & =2 \sum_{i,j} \frac{\partial_{i} m \partial^{3}_{ijj} m}{m}-\sum_{i,j} \partial^{2}_{jj} \left(\frac{ \left(\partial_{i} m \right)^{2}}{m}\right)-\sum_{i,j} \partial^{2}_{jj} m \frac{\left(\partial_{i} m \right)^{2}}{m^{2}},
\end{align*}
so that, since
\begin{align*}
\partial^{2}_{jj} \left(\frac{\left(\partial_{i} m\right)^{2}}{m} \right) & =\partial_{j} \left(2 \partial^{2}_{ij} m \frac{\partial_{i} m}{m}-\partial_{j} m \frac{\left(\partial_{i} m\right)^{2}}{m^{2}} \right)
\\ &
=2 \partial^{3}_{ijj} \frac{\partial_{i} m}{m}+2 \frac{\left(\partial^{2}_{ij} f \right)^{2}}{m}-2 \partial^{2}_{ij} m \frac{\partial_{i} m \partial_{j} m}{m^{2}}-\partial^{2}_{jj} m \frac{\left(\partial_{i} m \right)^{2}}{m^{2}}-2 \partial^{2}_{ij} m \frac{\partial_{i} m \partial_{j} m}{m^{2}}
\\ &
\quad +2 \frac{\left(\partial_{j} m \right)^{2} \left(\partial_{i} m \right)^{2}}{m^{3}},
\end{align*}
we simply have
\begin{align*}
A & = \sum_{i,j}\left(-2 \left(\partial^{2}_{ij} m \right)^{2}+4 \ \partial^{2}_{ij} m \frac{\partial_{i} m \partial_{j} m}{m^{2}}-2 \ \frac{\left(\partial_{j} m \right)^{2} \left(\partial_{i} m \right)^{2}}{m^{3}}\right)
\\ &
=-2 m \norm{\nabla^2 \log m}^2.
\end{align*}
Concerning $B$, noticing that
\begin{align*}
B= \sum_{i,j} 2 \ \frac{\partial_{i} m}{m} \partial^{2}_{ij} \left(b_{t}^{j} m\right)-\partial_{j} \left(b_{t}^{j} m \right) \frac{\left(\partial_{i} m\right)^{2}}{m^{2}}-\partial_{j} \left( b_{t}^{j} \frac{\left(\partial_{i} m\right)^{2}}{m} \right),
\end{align*}
the identities
\begin{align*}
\partial_{j} \left(b_{t}^{j} \frac{\left(\partial_{i} m\right)^{2}}{m} \right)=\partial_{j} b_{t}^{j} \ \frac{\left(\partial_{i} m \right)^{2}}{m}+2 b_{t}^{j} \ \partial^{2}_{ij} m \ \frac{\partial_{i} m}{m}-b_{t}^{j} \ \partial_{j} m \frac{\left(\partial_{i} f\right)^{2}}{f^{2}},
\end{align*}
and
\begin{align*}
\partial^{2}_{ij} \left(b_{t}^{j} m \right)=b_{t}^{j} \ \partial^{2}_{ij} m+m \ \partial^{2}_{ij} b_{t}^{j}+\partial_{i} b_{t}^{j} \ \partial_{j} m+\partial_{j} b_{t}^{j} \ \partial_{i} m,
\end{align*}
lead to the following upper bound:
\begin{align*}
B& = \sum_{i,j} 2 \frac{\partial_{i} m}{m} \partial_{ij}^{2} b_{t}^{j}+2 \partial_{i} b_{t}^{j} \frac{\partial_{i} m \partial_{j} m}{m}-\frac{\partial_{j} b_{t}^{j} \left( \partial_{i} m \right)^{2}}{m}
\\ &
\leq \sum_{i,j} \frac{\left(\partial_{i} m \right)^{2}}{m}+m \ \left(\partial_{ij}^{2} b_{t}^{j} \right)+\left(\partial_{i} b_{t}^{j} \right)^{2} \frac{\left(\partial_{i} m \right)^{2}}{m}+\frac{\left(\partial_{j} m\right)^{2}}{m}-\partial_{j} b_{t}^{j} \frac{\left(\partial_{i} m \right)^{2}}{m}
\\ &
\leq C \frac{\norm{\nabla m}^2}{m}+C m.
\end{align*}
Moreover, note that
\begin{align}
\nonumber
\left(\partial_{t} -\mathcal{L}\right) m \log m & =\partial_{t} m \ \log m+\partial_{t} m-\nabla \cdot \left( b_t \ m \log m\right)-\Delta \left( m \log m \right)
\\ &\nonumber
=\Delta m \log m+\nabla \cdot \left(b_t \ m \right) \log m+\Delta m +\nabla \cdot \left(b_t \ m \right)-\nabla \cdot b_t \ m \log m
\\ & \nonumber
\quad -b_t \cdot \nabla \left( m \log m \right)-\Delta \left(m \log m \right)
\\ & \nonumber
=\Delta m -m \ \Delta \log m-2 \frac{\norm{\nabla m}^{2}}{m}+ m \ \nabla \cdot b
\\ & \nonumber
=-\frac{\norm{\nabla m}^{2}}{m}+m \nabla \cdot b
\\ &
\leq -\frac{\norm{\nabla m}^{2}}{m}+C m\nonumber,
\end{align}
and that
\[
\left(\partial_{t}-\mathcal{L}\right) \left( (t+1) m \right)= (t+1) \left(\partial_{t}-\mathcal{L} \right)m+m= m.
\]
Therefore, for a positive constant $C_0$ large enough and a positive constant $C_1$ large enough (depending on $C_0$), we have
\begin{equation*} 
\left(\partial_{t}-\mathcal{L} \right) \left[ \frac{\norm{ \nabla m}^{2}}{m}+C_{0} m \log m- C_1 (t+1) m \right] \leq 0 .
\end{equation*}
Taking the values of $C_0$ and $C_1$ large enough, we also have
\[
\frac{\norm{ \nabla m_0}^{2}}{m_0}+C_{0} m_0 \log m_0- C_1 (t+1) m_0\leq 0.
\]
Indeed, $\log m_0(v)$ is supposed to have negative values for $v$ large enough, and the constant $c_{i,1}$ may be enlarged in \eqref{hyp:gassian tail II}. Note moreover that the Gaussian tail hypothesis for the initial condition \eqref{hyp:gaussian tail I} ensure (see \cite[Chapter 9]{friedman2008partial}) that for $t\in [0,T]$, 
\begin{equation}\label{eq:gaussian bounds parametrix}
C^{-1} e^{-c^{-1}\norm{v}^2}\leq m_t(v) \leq C e^{-c\norm{v}^2}, \quad \text{and}\quad \norm{\nabla m_t(v)}\leq C e^{c\norm{v}^2},
\end{equation}
for some constants $c$ and $C$ depending on $T$. We can then rely on a parabolic maximum principle on unbounded domain (see \cite[Theorem 3.4]{Meyer2014}) to deduce that, for all $t>0$, 
\begin{equation*}
\frac{\norm{\nabla m}^{2}}{m}+C_{0} m \log m\leq C_1 (t+1) m.
\end{equation*}
Dividing by $m$ and recalling \eqref{eq:gaussian bounds parametrix}, we deduce
\[\frac{\norm{\nabla m}^{2}}{m^2}(v) \leq C \left(1+\norm{v}^{2} \right). \]

\medskip

{\it Second step: bound on $\norm{\nabla^2 m_t}$.} The proof is similar to the first step.
Regrouping the terms involving only the Laplacian and the interaction terms we get after straightforward computations:
\begin{align*}
\left(\partial_{t}-\mathcal{L}\right) \frac{\norm{\nabla^{2} m}^{2}}{m} & =-\sum_{i,j,k} m \left(\frac{\partial^{3}_{ijk} m}{m}-\frac{\partial^{2}_{ij} m \partial_{k} m}{m^{2}} \right)^{2} +\sum_{ijk} \partial_{ijk}^{3} b^{k}_t \partial^{2}_{ij} m\\ 
&\quad +2\sum_{ijk} \frac{\partial^{2}_{ij} m}{m} \Bigl(\partial^{2}_{ij} b^{k}_t \partial_{k} m
+\partial^{2}_{ik} b^{k}_t \partial_{j} m+
\partial_{jk} b^{k}_t \partial_{i} m+\partial_{i} b^{k}_t \partial^{2}_{jk} m+\partial_{j} b^{k}_t \partial^{2}_{ik} m \Bigl).
\end{align*}
Relying then on the inequality $xy \leq \frac{x^{2}+y^{2}}{2}$, we deduce
\begin{equation*}
\left(\partial_{t}-\mathcal{L}\right) \frac{\norm{\nabla^{2} m}^{2}}{m} \leq C\frac{\norm{\nabla^{2} m}^{2}}{m}+C \frac{\norm{\nabla m}^{2}}{m}+C m .
\end{equation*}
Recall now that we obtained in the previous step the bounds
\begin{align*}
\left(\partial_{t}-\mathcal{L}\right) \frac{\norm{\nabla m}^{2}}{m} & \leq -2 m \sum_{i,j} \left(\frac{\partial^{2}_{ij} m}{m}-\frac{\partial_{i} m \partial_{j} m}{m^{2}} \right)^{2}+C \frac{\norm{\nabla m}^{2}}{m}+ C m
\\ &
\leq -2  \frac{\norm{\nabla^2 m}^{2}}{m}+C  \frac{\norm{\nabla m}^{4}}{m^{3}}+C \frac{\norm{\nabla m}^{2}}{m}+ C m,
\end{align*}
and
\[ 
\frac{\norm{\nabla m}^{2}}{m} \leq C_1 (t+1)- C_0 \log m,
\]
so that in particular
\[
\left(\partial_{t}-\mathcal{L}\right) \frac{\norm{\nabla m}^{2}}{m} \leq -2 \frac{\norm{\nabla^{2} m}^{2}}{m}-C_{0} \frac{\norm{\nabla m}^{2} \log m}{m}+C(t+1) \frac{\norm{\nabla m}^{2}}{m}+ C m.
\]
Therefore, for a constant $c$ large enough, we obtain
\begin{equation*}
\left(\partial_{t}-\mathcal{L}\right) \left[ \frac{\norm{\nabla^{2} m}^{2}}{m}+c  \frac{\norm{\nabla m}^{2}}{m} \right] \leq -c C_{0} \frac{\norm{\nabla m}^{2} \log m}{m}+C(t+1) \frac{\norm{\nabla m}^{2}}{m}+ C m.
\end{equation*}
Notice now that, for any positive constant $c$,
\begin{align*} 
\left(\partial_{t}-\mathcal{L} \right) \left(- m \left(\log m \right)^{2} +c(1+t) \ m \log m \right)  &= 2 \frac{\norm{\nabla m}^{2}}{m} \log m+ \left(2-c(1+t)\right) \frac{\norm{\nabla m}^{2}}{m}
\\ & \quad+\left( \nabla \cdot b+\cp{c} \right) m \log m+c(1+t)\nabla\cdot b  m,
\end{align*}
and that, since $\nabla \cdot b$ is bounded, for $c$ large enough $\nabla \cdot b+c \geq 0$. Using moreover the fact that $\log m$ is upper bounded by a constant, we get, for $c$ large enough,
\begin{equation*}
\left(\partial_{t}-\mathcal{L} \right) \left[-m \left(  \log m\right)^{2}+c(1+t)  m \log m \right] \leq 2 \log m \frac{\norm{\nabla m}^{2}}{m}+\left(2-c(1+t)\right) \frac{\norm{\nabla m}^{2}}{m}+(t+1) C m.
\end{equation*}
Lastly, notice that
\begin{align*}
    \left(\partial_{t}-\mathcal{L} \right) \left((1+t)^2 m\right) = 2(1+t) m,
\end{align*}
so that there exists $C_{3},C_{4},C_5,C_6$ (where $C_4$ may be enlarged, enlarging also $C_5$ and $C_6$), such that 
\begin{align*}
\left(\partial_{t}-\mathcal{L}\right) \left[\frac{\norm{\nabla^{2} m}^{2}}{m}+C_{3}  \frac{\norm{\nabla m}^{2}}{m}-C_{4} m \left(\log m \right)^{2} +C_{5} (t+1) m \log m-C_{6} (t+1)^{2} m \right] \leq 0.
\end{align*}
As in the first step, $c_{i,2}$ may be enlarged if needed in \eqref{hyp:gassian tail II}, so the constants $C_{3},C_{4},C_5,C_6$ may be chosen so that
\[
\frac{\norm{\nabla^{2} m_0}^{2}}{m_0}+C_{3}  \frac{\norm{\nabla m_0}^{2}}{m_0}-C_{4} m_0 \left(\log m \right)^{2} +C_{5} (t+1) m_0 \log m_0-C_{6} (t+1)^{2} m_0\leq 0. 
\]
Relying again on a maximum principle we deduce that for $t\in [0,T]$,
\begin{align}\label{eq:max principle case 2}
\frac{\norm{\nabla^{2} m}^{2}}{m} & \leq C_{6} (t+1)^{2} m- C_{5}(1+ t)m \log m + C_4 m (\log m)^2,
\end{align}
and, recalling \eqref{eq:gaussian bounds parametrix},
\[
\frac{\norm{\nabla^{2} m}^{2}}{m^2}(v) \leq C\left(1+\norm{v}^4\right).
\]

\medskip
{\it Third step: bound on $\norm{\nabla^2 m}$.}
With a very similar approach to the previous cases (by separating the terms involving the laplacian and the interaction), we obtain
\begin{align}
\nonumber \left(\partial_{t}-\mathcal{L}\right) \frac{\norm{\nabla^{3} m}^{2}}{f} & \leq -2 \sum_{i,j,k,l} m \left( \frac{\partial^{4}_{ijkl} m}{m}-\frac{\partial^{3}_{ijk} m \partial_{l} m}{m^{2}}\right)^{2}\\
&\quad +C\left(\frac{\norm{\nabla^{3} m}^{2}}{m}+\frac{\norm{\nabla^{2} m}^{2}}{m}+\frac{\norm{\nabla m}^{2}}{m}+m\right)
\nonumber\\ & \leq C\left(\frac{\norm{\nabla^{3} m}^{2}}{m}+\frac{\norm{\nabla^{2} m}^{2}}{m}+\frac{\norm{\nabla m}^{2}}{m}+m\right). \label{eq_grad3}
\end{align}
Similarly to the previous case, we rely on the dissipation term produced by $\frac{\norm{\nabla^{2} m}^{2}}{m}$:
\begin{align*}
\left(\partial_{t}-\mathcal{L}\right) \frac{\norm{\nabla^{2} m}^{2}}{m}  & \leq -2 \sum_{i,j,k} m \left(\frac{\partial^{3}_{ijk} m}{m}-\frac{\partial^{2}_{ij} m \partial_{k} m}{m^{2}} \right)^{2}+ C\frac{\norm{\nabla^{2} m}^{2}}{m}+C \frac{\norm{\nabla m}^{2}}{m}+C m
\\ &
\leq -2 \sum_{i,j,k} \frac{\left( \partial^{3}_{ijk} m \right)^{2}}{m}+C \frac{\norm{\nabla^2 m}^2\norm{\nabla m}^2}{m^3}+C\frac{\norm{\nabla^{2} m}^{2}}{m}+C \frac{\norm{\nabla m}^{2}}{m}+C m.
\end{align*}
Relying on \eqref{eq:max principle case 2} we then deduce
\begin{align*}
\left(\partial_{t}-\mathcal{L}\right) \frac{\norm{\nabla^{2} m}^{2}}{m} & \leq 
-2 \sum_{i,j,k} \frac{\left( \partial^{3}_{ijk} m \right)^{2}}{m}-C_{5} (t+1)\frac{\norm{ \nabla m}^2}{m} \log m
\\ &
\qquad 
 +C\frac{\norm{\nabla^{2} m}^{2}}{m}+C(t+1)^{2} \frac{\norm{\nabla m}^{2}}{m}+C m+C_{4} \frac{\norm{\nabla m}^{2}}{m} \left(\log m \right)^{2}.
\end{align*}
Notice now that for positive constants $c$ and $c'$, with $c'$ taken large enough,
\begin{align*}
\left(\partial_{t}-\mathcal{L} \right) &\left[-c(t+1) m \left( \log m \right)^{2}+c' \frac{(t+1)^{2}}{2} m \log m \right] \\
&=
-c m \left(\log m \right)^{2}+2c(t+1) \frac{\norm{\nabla m}^{2}}{m} \log m+2c (t+1) \frac{\norm{\nabla m}^{2}}{m}
 -c'\frac{(t+1)^{2}}{2} \frac{\norm{\nabla m}^{2}}{m} \\
&\quad +c' \frac{(t+1)^{2}}{2}  m+t \left( \nabla \cdot b+c'\right) m \log m
\\ &
\leq 2c(t+1) \frac{\norm{\nabla m}^{2}}{m} \log m+2c(t+1) \frac{\norm{\nabla m}^{2}}{m}+c' \frac{\left(t+1\right)^{2}}{2} m-cm \left( \log m \right)^{2},
\end{align*}
so that, choosing $c=C_5/2$, there exists a constant $C_7$ such that
\begin{align*}
\left( \partial_{t}-\mathcal{L} \right)& \left[\frac{\norm{\nabla^{2} m}^{2}}{m}-\frac{C_5}{2} (t+1) m \left( \log m \right)^{2}+C_7 \frac{(t+1)^{2}}{2} m \log m \right] \\
& \leq -2 \frac{\norm{\nabla^{3} m}^{2}}{m}
 +C (t+1)^{2} \frac{\norm{\nabla m}^{2}}{m}
+C (t+1)^{2} m- \frac{C_5}{2}  m \left( \log m \right)^{2}+C \frac{\norm{\nabla m}^{2}}{m} \left( \log m \right)^{2}.
\end{align*}
The only term that poses a difficulty here is the last one. We can take care of it by evaluating
\begin{align*}
\left(\partial_{t}-\mathcal{L} \right) \left[ m \left( \log m \right)^{3} \right]=3 \nabla \cdot b m \left(\log m \right)^{2}
-3 \frac{\norm{\nabla m}^{2}}{m}  \left(\log m \right)^{2}-6 \frac{\norm{\nabla m}^{2}}{m} \log m.
\end{align*}
There exists $C_{8}$ such that 
\begin{align*}
\left(\partial_{t}-\mathcal{L} \right)& \Bigg[\frac{\norm{\nabla^{2} m}^{2}}{m}-\frac{C_{5}}{2} (t+1) m \left( \log m \right)^{2}+C_{7} \frac{(t+1)^{2}}{2} m \log m +C_{8} m \left( \log m\right)^{3} \Bigg] \\
&\leq -2 \frac{\norm{\nabla^{3} m}^{2}}{m}+C(t+1)^{2} \frac{\norm{\nabla m}^{2}}{m}
+C(t+1)^{2} m+\left(3 \nabla \cdot b-\frac{C_{5}}{2} \right) m \left(\log m \right)^{2}\\
&\quad -C \frac{\norm{\nabla m}^{2}}{m} \log m.
\end{align*}
Note moreover that the value of the constant $C_{5}$ defined before may be enlarged such that $3 \nabla \cdot b-\frac{C_{5}}{2} \leq 0$.
The term $\frac{\norm{\nabla^{3}m}^{2}}{m}$ in \eqref{eq_grad3} can then be dealt with, and proceeding as before we can find positive constants $C_{9}, C_{10}, C_{11}, C_{12},C_{13}, C_{14}$ such that 
\begin{multline*}
\left(\partial_{t}-\mathcal{L} \right) \Bigg[\frac{\norm{\nabla^{3} m}^{2}}{m}+ C_{9}\frac{\norm{\nabla^{2} m}^{2}}{m}+C_{10}\frac{\norm{\nabla m}^{2}}{m}+C_{11} (t+1)^{2} m\log m
 \\ -C_{12} (t+1) m \left(\log m\right)^{2}+C_{13} m \left( \log m \right)^{3}-C_{14} (t+1)^{3} m\Bigg]  \leq 0.
\end{multline*}
We can then conclude the proof of Lemma \ref{lemme_grad} by a maximum principle as in the previous steps.

\end{proof}

\begin{proof}[Proof of Proposition~\ref{prop:Ik...Nk}]
The Gaussian tail hypothesis for the initial condition \eqref{hyp:gaussian tail I} ensures (see \cite[Chapter 9]{friedman2008partial}) that for $t\in [0,T]$, $i=1,2$ and $k=1,\ldots,N$,
\begin{equation}
 m^k_{i,t}(v) \leq C e^{-c\norm{v}^2}.
\end{equation}
Together with Lemma~\ref{lemme_grad} this implies directly that 
\[
\max_{k=1,\ldots,N} \sup_{t\in [0,T]} (M^k_t+N^k_t+O^k_t)<\infty.
\]
Due to Lemma~\ref{lem:hierarchy I J K} and Lemma~\ref{lem:link Lk Jk Nk} it remains to prove that $I^N_t$, $J_t^N$ and $K_t^N$ are finite.
Since the proofs are very similar, let us focus on $J^{N}$ and $K^{N}$ only. We have
\begin{align*}
J^{N}& =\int m_{2}^{N} \frac{\norm{\nabla^{2} f_{N}}^{2}}{f_{N}}
\leq C \int m_{1}^{N} \norm{\nabla^{2} \log f_{N}}^{2}+C\int m_{1}^{N} \norm{ \nabla \log f_{N}}^{4}.
\end{align*}
Using the decomposition $\log f_{N}=\log m_{1}^{N}-\log m_{2}^{N}$ and lemma~\ref{lemme_grad}, we get:
\begin{align*}
J^{N} \leq C+C \int m_{1}^{N}(dv) \norm{v}^{4} \leq C.
\end{align*}
The proof for $K^{N}$ is similar, relying on the decomposition
\begin{align*}
\frac{\partial^{3}_{i,j,k} f_N}{f_N}&=\partial^{3}_{i,j,k} \log f_N+\partial_{i} \log f_N \partial^{2}_{j,k} \log f_N+\partial_{j} \log f_N \partial^{2}_{i,k} \log f_N+\partial_{k} \log f_N \partial^{2}_{i,j} \log f_N\\
&\quad +\partial_{i} \log f_N \partial_{j} \log f_N \partial_{k} \log f_N.
\end{align*}
\end{proof}

\subsection{Proof of Proposition~\ref{prop:derivatives b1}}

Let us focus on the first bound of Proposition~\ref{prop:derivatives b1}. The other bounds can be obtained in a similar way. We have
\begin{align*}
    \int m^k_2 \nabla_{v_i} f_k \cdot \nabla_{v_i} \left(\nabla_{v_j} \cdot b^{j,k}_1\right) &= \frac{1}{N}\sum_{l=1}^k \int m^k_2 \nabla_{v_i} f_k \cdot \nabla_{v_i} \left(\nabla_{v_j} \cdot \Gamma \right)(v_j-v_l)\\
    &\quad + \frac{N-k}{N}\int m^k_2 \nabla_{v_i} f_k \cdot \nabla_{v_i} \left(\nabla_{v_j} \cdot \left(\Gamma(v_j-v_{k+1}) m_1^{k+1|k}\right)\right).
    \\ &
    \quad+ \int m_{2}^{k} \nabla_{v_{i}} f_{k} \cdot \nabla_{v_{i}} \left(\nabla \cdot F(v_{j}) \right).
\end{align*}
Since the derivatives of $\Gamma$ are bounded we directly have
\begin{equation*}
    \left|\int m^k_2 \nabla_{v_i} f_k \cdot \nabla_{v_i} \left(\nabla_{v_j} \cdot \Gamma \right)(v_j-v_l)\right| \leq C I^k.
\end{equation*}
The third term is only non zero if $i=j$, and we have, in that case, because the derivatives of F are bounded:
\begin{align*}
 \left|\int m_{2}^{k} \nabla_{v_{i}} f_{k} \nabla \left( \nabla \cdot F \right)(v_{i}) \right| & \leq C  \int m_{2}^{k} \norm{\nabla_{v_{i}} f_{k}} 
 \\ &
 \leq C I_{t}^{k}+C
\end{align*}
For the second term, we rely on the following decomposition:
\begin{align*}
    \int m^k_2 \nabla_{v_i} f_k \cdot \nabla_{v_i} \left(\nabla_{v_j} \cdot \left(\Gamma(v_j-v_{k+1}) m_1^{k+1|k}\right)\right) &
    =\int m^k_2 \nabla_{v_i} f_k \cdot \nabla_{v_i} \left(\nabla_{v_j} \cdot \Gamma\right)(v_j-v_{k+1})\,  m_1^{k+1|k}\\
    &\quad + \int m^k_2 \nabla_{v_i} f_k \cdot \nabla_{v_i}   m_1^{k+1|k} \left(\nabla_{v_j} \cdot \Gamma\right) (v_j-v_{k+1}) \\
    &\quad + \int m^k_2 \nabla_{v_i} f_k \cdot \nabla_{v_i}  \Gamma (v_j-v_{k+1}) \nabla_{v_j}   m_1^{k+1|k}\\
    &\quad + \int m^k_2 \nabla_{v_i} f_k \cdot  \nabla^2_{v_i,v_j}   m_1^{k+1|k}  \Gamma (v_j-v_{k+1}).
\end{align*}
We treat each of the terms of the right-hand side separately. For the first one, we have
\begin{align*}
\left|\int m^k_2 \nabla_{v_i} f_k \cdot \nabla_{v_i} \left(\nabla_{v_j} \cdot \Gamma\right)(v_j-v_{k+1})\,  m_1^{k+1|k}\right| \leq C\int m^k_2 \Vert \nabla_{v_i} f_k \Vert \int  m_1^{k+1|k}\leq C I^k.
\end{align*}
For the second one, we get
\begin{align*}
    &\left|\int m^k_2 \nabla_{v_i} f_k \cdot \nabla_{v_i}   m_1^{k+1|k} \left(\nabla_{v_j} \cdot \Gamma\right) (v_j-v_{k+1}) \right|\\
    &\qquad= \left|\int m^{k+1}_1 \nabla_{v_i} \log f_k \cdot \nabla_{v_i}  \log m_1^{k+1|k}\left(\nabla_{v_j} \cdot \Gamma\right) (v_j-v_{k+1}) \right|\\
    &\qquad \leq C I^k + C\int m^{k+1}_1 \left\Vert \nabla_{v_i} \log m_1^{k+1|k}\right\Vert^2\\
    &\qquad \leq C I^k + C I^{k+1}.
\end{align*}
The third term can be treated in the same way, and for the last one we have
\begin{align*}
    &\left|\int m^k_2 \nabla_{v_i} f_k \cdot  \nabla^2_{v_i,v_j}   m_1^{k+1|k}  \Gamma (v_j-v_{k+1})\right|\\
    &\qquad = \left|\int m^{k+1}_1 \nabla_{v_i}\log f_k \cdot \left(\nabla^2_{v_i,v_j} \log m_1^{k+1|k} +\nabla_{v_i}\log m_1^{k+1|k} \otimes \nabla_{v_j}\log m_1^{k+1|k}\right)\Gamma (v_j-v_{k+1})\right|\\
    &\qquad \leq C I^k +C\int m^{k+1}_1 \left\Vert \nabla^2 \log m_1^{k+1|k}\right\Vert^2+C\int m^{k+1}_1 \left\Vert \nabla \log m_1^{k+1|k}\right\Vert^4,
\end{align*}
and relying on the fact that $\log m_1^{k+1|k}= \log f_{k+1}-\log f_k+\log m^{k+1}_2-\log m^k_2$ we get
\begin{equation*}
    \int m^{k+1}_1 \left\Vert \nabla \log m_1^{k+1|k}\right\Vert^4\leq C\left( L^k+ L^{k+1} +  M^k +  M^{k+1}\right),
\end{equation*}
and 
\begin{equation*}
    \int m^{k+1}_1 \left\Vert \nabla^2 \log m_1^{k+1|k}\right\Vert^2\leq C\left(J^k + J^{k+1} + L^k + L^{k+1}+ N^k +  N^{k+1}\right).
\end{equation*}
This concludes the proof for the first bound of Proposition~\ref{prop:derivatives b1}, recalling Lemma~\ref{lem:hierarchy I J K}.

\subsection{Proof of Proposition~\ref{LemDev2}}

Under our assumptions, we know that $b_{2,t}$ (recall \eqref{def:b2}) satisfies $b_{2,t}=\Gamma \ast \bar\mu_{t}+F=\nabla U_t$, with all derivatives of $\nabla U_t$ bounded.
Therefore, $\bar{\mu}_{t}$ satisfies the equation
\begin{align*} \partial_{t} \bar\mu_{t}& =\text{Tr}\left( \nabla^{2} \bar\mu_{t} \right)-\nabla \cdot \left(b_{2,t} \bar \mu_{t} \right)
\\ &
=\text{Tr}\left( \nabla^{2} \bar\mu_{t} \right)-\left(\nabla \cdot b_{2,t}\right) \ \bar \mu_{t}-\nabla U_{t} \cdot \nabla \bar \mu_{t}
\\ &
=\text{Tr}\left( \nabla^{2} \bar\mu_{t} \right)-V \bar \mu_{t}-\nabla U_{t} \cdot \nabla \bar \mu_{t},
\end{align*}
where $\nabla U_{t}=\Gamma \ast \bar\mu_{t}+F$ and $V=\nabla \cdot \left( \Gamma \ast \bar \mu_{t}+F\right)$
This means that $\bar\mu_{t}$ satisfies the same equation as $\psi_{t}$ in section 2 of \cite{bdd_hess}, so $-\log \bar\mu_{T-t}$ satisfies the same equation as $\phi_{t}$. Using the same notations as \cite{bdd_hess}, we have $h=-\log \bar\mu_{0}$, so that $\nabla^{2} h$ is bounded (and thus $\nabla h$ is uniformly lipschitz). We also know that $\nabla U= \Gamma \ast \bar\mu_{t}+F$ is uniformly lipschitz. Notice moreover that, because our coefficients may depend on time, we have to consider a generalized version of the operator $\mathcal{A}_{U}$, so we need to check that $\nabla \left[\mathcal{A}_{U}+V\right]=\nabla \left[ \frac{1}{2} | \nabla U |^{2}-\frac{1}{2} \Delta U+V \right]-\nabla \left(\partial_{t} U \right)$ is uniformly lipschitz. Because $\Gamma$ and all the derivatives of $\Gamma $ and $F$ are uniformly bounded, $\nabla \left[ \frac{1}{2} | \nabla U |^{2}-\frac{1}{2} \Delta U+V \right]$ is uniformly lipschitz. Moreover, we have
\begin{align*} \nabla \left(\partial_{t} U \right) & =\partial_{t} \nabla U 
\\ &
=\Gamma \ast \partial_{t} \bar\mu_{t}+\partial_{t} F
\\ &
=\Gamma \ast \Delta \bar\mu_{t}-\Gamma \ast \left( \nabla \cdot \left[ \Gamma \ast \bar\mu_{t} \ \bar\mu_{t} \right] \right)
\\ &
=\Delta \Gamma \ast \bar\mu_{t}-\nabla \cdot \Gamma \ast \left(\Gamma \ast \bar\mu_{t} \ \bar\mu_{t} \right)
\end{align*}
Because $\Gamma$ and its derivatives are bounded, this identity entails that $\nabla \partial_{t} U$ is uniformly lipschitz.
We also have, because $\mathcal{A}+V$ is bounded, that  \[\kappa_{\mathcal{A}_{t}+V_{t}}=\frac{1}{r^{2}} \sup_{|x-y|=r} \langle \nabla \mathcal{A}_{t}(x)+V_{t}(x)-\mathcal{A}_{t}(y)-V_{t}(y),x-y\rangle \geq -\frac{C}{r} \]
According to the first point of Section 2.2, we can take $f_{t}(r)=-\frac{r}{t+\alpha}+\frac{\beta_{1}}{2} (t+\alpha)-\frac{\beta_{2}}{t+\alpha}$ for a suitable choice of $\beta_{1}, \beta_{2}$ and $\alpha$ with which we can apply Corollay 2.5 of \cite{bdd_hess} with $\sigma=\text{Id}$. This entails that $\nabla^{2} \log \mu_{T-t}$ is bounded for all $t,T$. This concludes the proof of Lemma 2.5.

\medskip

\noindent{\bf Acknowledgments.}\\
This work has been (partially) supported by the Project CONVIVIALITY ANR-23-CE40-0003 of the French National Research Agency. AG has benefited from a government grant managed by the Agence Nationale de la
Recherche under the France 2030 investment plan ANR-23-EXMA-0001.

\nocite{*}
\bibliographystyle{abbrv}
\bibliography{mabiblio}

\end{document}